\documentclass[a4paper,11pt]{article}
\usepackage{amsthm,amsmath,amssymb,amsfonts,mathrsfs,hyperref,microtype}
\usepackage{graphicx,color}
\usepackage{eufrak}

\newtheorem{defn}{Definition}[section]
\newtheorem{thm}[defn]{Theorem}
\newtheorem{pro}[defn]{Proposition}
\newtheorem{cor}[defn]{Corollary}
\newtheorem{rem}[defn]{Remark}
\newtheorem{lma}[defn]{Lemma}

\newtheorem{exa}[defn]{Example}
\def\N{{\rm I\kern-0.16em N}}
\def\R{{\rm I\kern-0.16em R}}
\def\E{{\rm I\kern-0.16em E}}
\def\P{{\rm I\kern-0.16em P}}
\def\F{{\rm I\kern-0.16em F}}
\def\B{{\rm I\kern-0.16em B}}
\def\C{{\rm I\kern-0.46em C}}
\def\G{{\rm I\kern-0.50em G}}

\numberwithin{equation}{section}
\font\eka=cmex10
\usepackage{ae}
\def\ind{\mathrel{\hbox{\rlap{%
\hbox to 7.5pt{\hrulefill}}\raise6.6pt\hbox{\eka\char'167}}}}
\newcommand{\vr}{L_{[0,r]}}     
\newcommand{\vT}{L_{[r,T]}}
\newcommand{\vro}{L_{[0,r]}^{\perp}}
\newcommand{\vTo}{L_{[r,T]}^{\perp}}

\newcommand{\la}{\langle}   
\newcommand{\ra}{\rangle}  
\newcommand{\ud}{\mathrm{d}}    

\newcommand{\fc}{h}  
\newcommand{\Ae}{A}			
\newcommand{\inth}{\underline{h}} 
\newcommand{\cinth}{\underline{c}} 

\def\HH{ \EuFrak H}			
\newcommand{\Cset}{\mathcal{A}}    

\newcommand{\I}{\mathbb{I}} 
\newcommand{\1}{\mathbf{1}} 

\newcommand{\Tr}{\mathbb{I}} 

\newcommand{\ff}{\EuFrak f} 
\newcommand{\hh}{\EuFrak h}
\renewcommand{\gg}{\EuFrak g}  
\newcommand{\cc}{\EuFrak c} 
\newcommand{\CC}{\EuFrak C}	

\begin{document}
\title{\textbf{A general non-existence result for linear BSDEs driven by Gaussian processes}}
\author{Christian Bender \and Lauri Viitasaari}

\renewcommand{\thefootnote}{\fnsymbol{footnote}}

\author{Christian Bender\footnotemark[1] \, and \, Lauri Viitasaari\footnotemark[1] \footnotemark[2]}

\footnotetext[1]{Department of Mathematics, Saarland University, Saarbr\"ucken\\
Postfach 151150, D-66041 Saarbr\"ucken, Germany}

\footnotetext[2]{Department of Mathematics and System Analysis, Aalto University School of Science, Helsinki P.O. Box 11100, FIN-00076 Aalto, Finland}

\maketitle

\begin{abstract}

In this paper, we study linear backward stochastic differential equations driven by a class 
of centered Gaussian non-martingales, including fractional Brownian motion with Hurst parameter $H\in (0,1)\setminus \{\frac12\}$. We show that, for every 
choice of deterministic coefficient functions, there is a square integrable terminal condition such that
the equation has no solution.

\medskip


\small 

\end{abstract}

\section{Introduction}

In this paper, we study linear backward stochastic differential equations (BSDEs) driven by a centered Gaussian process $X$. More, precisely, we consider a BSDE of
 the form 
\begin{equation}\label{LBSDE_intro}
 \ud Y_t = \left(a(t)Y_t + G_t\right)\ud\gamma(t) +  Z \ud \cinth(t) +  Z \ud^{\diamond} X_t,\;0\leq t\leq T, \quad Y_T=\xi,
\end{equation}
for deterministic functions $a, \gamma, \cinth$, where $\gamma$ is a continuous function of bounded variation and $\cinth$ belongs to the Cameron-Martin space of $X$.
Moreover, $G$ is an adapted process, which satisfies suitable
integrability assumptions.  The diamond in equation (\ref{LBSDE_intro}) indicates that 
a (generalized) Skorokhod integral is applied in the integral form of the equation. On the Gaussian process $X$ we merely assume that one can define indefinite Wiener integrals 
with respect to $X$ and that 
these indefinite Wiener integrals are continuous with respect to the integrand (see Definition \ref{def:indefinite} below). This property holds true, e.g., for fractional Brownian motions with arbitrary
Hurst parameter $H\in (0,1)$, which is our leading example throughout the paper. 
As a main result we show the following general non-existence result: Whenever $X$ is not a martingale, then, for every choice of the coefficients  $a,\gamma,\cinth$, and $G$, there is a square integrable terminal condition $\xi$ such that
the above BSDE has no solution (in a mild sense, as defined precisely in Section \ref{sec:BSDE}). Put differently, well-posedness of linear BSDEs driven by centered Gaussian processes can only hold true in the martingale case.

In the fractional Brownian motion case, linear BSDEs with deterministic coefficients were previously studied 
by  \cite{bia-hu-et-al}, \cite{bender-05a}, and \cite{bender-05b}. On the one hand, combining the results from \cite{bender-05a} and \cite{bender-05b} yields an existence and uniqueness result for linear 
fractional BSDEs with deterministic coefficients for the full range of Hurst parameters $H\in (0,1)$ provided the terminal condition is a deterministic function applied to the driving fractional 
Brownian motion $X_T$ at terminal time.  On the other hand, the authors of \cite{bia-hu-et-al} consider the case of general square integrable terminal conditions for fractional Brownian motion with Hurst parameter $H>1/2$,
and represent the $Y$-part of the solution in terms of the quasi-conditional expectation operator introduced by \cite{hu-oksendal}. This quasi-conditional expectation bears many similarities 
with classical conditional expectation, but it has been realized in recent years that there are also some subtleties related to this operator. In the case of fractional Brownian motion with Hurst parameter 
$H>1/2$, it is shown by \cite{hu-peng} that the monotonicity property and the `taking-out-what-is known'-property fail to hold for the corresponding quasi-conditional expectation operator, while
the authors of \cite{bender-elliott} construct a counterexample showing that Jensen's inequality is also not in force for quasi-conditional expectation. With these subtleties in mind, a careful analysis 
of the arguments in \cite{bia-hu-et-al} reveals that their existence result requires additional assumptions. Nonetheless, \cite{bia-hu-et-al} makes for the first time the important connection between linear fractional 
BSDEs and quasi-conditional expectation, which, in generalized form, is also at the core of our non-existence result.

Let us finally mention that the interest in BSDEs driven by a fractional Brownian motion was revived by the work of \cite{hu-peng}, where, for the first time, nonlinear fractional BSDEs under a Lipschitz 
condition are considered. In \cite{hu-peng} and some recent related generalizations such as \cite{mat-nie} and  \cite{fei-xia-zhang}, the randomness in the coefficients and the terminal condition is basically 
given by a deterministic function applied to the driving process. As shown in \cite{bender} in such cases the solution can always be constructed directly on a functional level, 
by considering an auxiliary Markovian BSDE driven by a classical Brownian motion and performing a deterministic time change. 

The present paper is structured as follows: In Section 2, we introduce several generalizations of Skorokhod integrals for centered Gaussian processes. The basic idea is to start with the classical 
definition of the Skorokhod integral for a class of simple integrands, but to observe that this integral can be extended to a densely defined closed operator with respect to several different Hilbert space norms. 
Our construction is based on the $S$-transform, and covers, besides the classical Skorokhod integral, the extended divergence operator of \cite{nualartleon} and 
the Wick-It\^o integral of \cite{bender}.

 In Section 3, we first discuss our notion of a mild solution to (\ref{LBSDE_intro}) and then present our main results on linear BSDEs driven by
Gaussian processes. Compared to classical solutions to BSDEs, our notion of a mild solution adds more flexibility in two respects. Firstly, we do not ask the $Z$-part of the solution to be adapted, 
while, secondly and more importantly, in the integral form of (\ref{LBSDE_intro}), one can freely choose among the several generalizations of the Skorokhod integral from Section 2. We emphasize that 
it is well-known from \cite{cheriditonualart} that, e.g. in the case of fractional Brownian motion with Hurst parameter $H\leq 1/4$, the domain of the classical Skorokhod integral is so small, that its 
does not even contain fractional Brownian motion itself. Allowing to choose among the several Skorokhod type integrals with different domains is therefore important to 
ensure that our non-existence result cannot be blamed to be a consequence of using an `inappropriate' variant of the Skorokhod integral with too small a domain. The proof of our main result on non-existence 
relies on two observations: (i) Given a mild solution, the $Y$-part of the solution can necessarily be represented in terms of a shifted quasi-conditional expectation operator; and (ii) 
The domain of this shifted quasi-conditional expectation operator is a true subspace of the space of square-integrable random variables, whenever $X$ is not a martingale.

In order to prove these two auxiliary results, we need to introduce and study shifted quasi-conditional expectation in some detail in Section 4. We provide a new definition of shifted quasi-conditional expectation 
(in terms of the $S$-transform) which acts as a densely defined closed operator on the $L^2$-space generated by the information of a centered Gaussian process $X$, under our standing assumption that $X$ has an 
indefinite Wiener integral. In the case of a fractional Brownian motion with Hurst parameter $H>1/2$, quasi-conditional expectation was previously defined in terms of the chaos expansion
 in \cite{hu-oksendal}, see also 
\cite{bender-elliott} and \cite{hu-peng}. It turns out that even in this case our new definition leads to a larger domain than the one in these papers. E.g., all random variables with finite chaos are included
in the domain 
in our setting without requiring that the chaos coefficients are distributions of function type. A complete characterization of the domain of shifted quasi-conditional expectation in terms of the 
chaos decomposition  
is provided in Theorem \ref{thm:domain}. Together with the observation that the restriction of quasi-conditional expectation to the first chaos has operator norm stricly larger than one, if and only if $X$ is not a martingale,
this characterization result shows that the domain of shifted quasi-conditional expectation is a true subset of the space of $L^2$-random variables in the non-martingale case.

In the final Section 5, we explain how to check our standing assumption that $X$ has an indefinite Wiener integral.
To this end we reformulate this assumption equivalently 
in terms of the Cameron-Martin space of $X$. This reformulation gives rise to the fact that this standing assumption is satisfied by (linear combinations of independent) fractional 
Brownian motions, which can be seen by applying well-known results on fractional integrals.


\section{$\mathcal{H}$-Skorokhod integration}
In this section, we introduce a Skorokhod type integral with respect to a Gaussian process $X$. The main emphasis is on different ways to extend the integral from simple integrands to 
larger sets of integrands. Our setting covers the classical Skorokhod integral as well as the Wick-It\^o integral in the sense of \cite{bender}.

Throughout the paper we assume that $X=(X_t)_{t\in[0,T]}$ is a measurable centered Gaussian process with $X_0=0$ on a compact interval $[0,T]$ living in a complete probability space $(\Omega,\mathcal{F},P)$. We denote by $\left(\mathcal{F}_t^X\right)_{t\in[0,T]}$ the filtration generated by $X$ and set $L^2_X = L^2(\Omega,\mathcal{F}_T^X,P)$.
The first chaos associated to $X$ is
$$
H_X = \overline{\text{span}\{X_s : 0\leq s \leq T\}},
$$
where the closure is taken in $L^2_X$. Note that each element $\fc \in H_X$ is a centered Gaussian random variable. We recall that elements $\fc \in H_X$ can be thought of as  integrals of deterministic integrands with respect 
to $X$. To this end 
we denote by $\HH_X$ the Hilbert space which is obtained as closure of the linear span of the indicator functions $\textbf{1}_{(0,t]},\; t\in[0,T]$, equipped with the inner product
$$
\la \textbf{1}_{(0,t]}, \textbf{1}_{(0,s]} \ra_{\HH_X} = R(t,s),
$$
where $R(t,s):=\E[X_t X_s]$ is the covariance function of the process $X$. $\HH_X$ is called the \emph{space of deterministic integrands} associated to $X$. Then, the mapping
$$
\textbf{1}_{(0,t]} \mapsto X_t
$$ 
extends in a unique way to a linear isometry $I: \HH_X \rightarrow H_X$, which is called \emph{Wiener integral}. Since, by definition, the set $\{X_t:\, 0\leq t \leq T\}$ is total in the first 
chaos $H_X$, the map 
$$
\inth: [0,T] \rightarrow \mathbb{R},\quad t\mapsto \E[X_t \fc]
$$
uniquely characterizes the random variable $\fc$, i.e. if $\E[X_t \fc_1]=\E[X_t \fc_2]$ for every $t\in [0,T]$, for elements $\fc_1, \fc_2 \in H_X$, then $\fc_1=\fc_2$ almost surely. 
If we endow the set of functions of the form $\inth=\E[X_\cdot \fc]$ for $\fc\in H_X$ with the inner product $\la \inth_1,\inth_2 \ra_{CM_X}=  \E[\fc_1 \fc_2]$, we hence obtain a Hilbert space, which is also isometric 
to $H_X$, and known as the \emph{Cameron-Martin space} of $X$. We denote it by $CM_X$.
Throughout the paper we denote  generic elements of the first chaos $H_X$ by italic letters $\fc$ and the associated elements in the Cameron-Martin space $CM_X$ by underlined 
letters $\inth$ and 
in the space of deterministic integrands $\HH_X$ by fractur letters $\hh=I^{-1}(\fc)$. We sometimes remove the subscript $X$, if no confusion can arise.

Recall that random variables of form 
$$
e^{\diamond \fc} := \text{exp}\left\{\fc - \frac12 \E \fc^2\right\} ,\quad \fc\in H,
$$
are called \emph{Wick-exponentials}. Whenever $A$ is a dense subset of $H$, the set $\{e^{\diamond \fc},\; \fc\in A  \}$ forms a total subset of $L^2_X$, see e.g. Corollary 3.40 in \cite{janson}.
Consequently, every random variable $\xi$ is uniquely determined by its $S$-transform defined by
$$
\left(S\xi\right)(\hh) = \E\left[\xi e^{\diamond I(\hh)}\right],\quad \hh\in \HH.
$$ 
More precisely, if $\left(S\xi_1\right)(\hh) = \left(S\xi_2\right)(\hh)$ for every $\hh\in \Ae \subset \HH$ from a dense subset $\Ae$, then $\xi_1 = \xi_2$ almost surely. Note also that for any element $g \in H_X$ we have
\begin{equation}
\label{eq:S_transform_first_chaos}
\left(Sg\right)(\hh) = \E\left[g e^{\diamond I(\hh)}\right] = \E\left[g I(\hh)\right].
\end{equation}
In particular, for $X_t$ with $t\in[0,T]$, we obtain
\begin{equation} \label{eq:S_transform_first_chaos2}
\left(S X_\cdot\right)(\hh) = \E\left[ X_\cdot I(\hh)\right]=\inth \in CM.
\end{equation}
We finally note the well-known simple identities
\begin{equation}\label{eq:S_transform_properties}
 \E[ e^{\diamond \fc}]=1,\quad e^{\diamond \fc}e^{\diamond g}=e^{\E[gh]}e^{\diamond (\fc+g)}
\end{equation}
for $\fc, g \in H$.

With this notation at hand, we can now turn to the definition of different notions of Skorokhod type integrals with respect to $X$. First, we denote by 
$\mathcal{E}$ the space of \emph{simple integrands} defined by 
$$
\mathcal{E} =\text{span}\{\textbf{1}_{(a,b]}(t) e^{\diamond \fc},\quad 0\leq a<b\leq T,\fc \in H_X\}.
$$
The Skorokhod integral of an element $Z_t = \sum_{k=1}^n \textbf{1}_{(a_k,b_k]}(t) e^{\diamond \fc_k} \in \mathcal{E}$ is defined as
\begin{equation}
\label{eq:simple_integral}
\int_0^T Z_s \ud^\diamond X_s = \sum_{k=1}^n e^{\diamond \fc_k}\left(X_{b_k} - X_{a_k}\right) - \sum_{k=1}^n \E\left[\fc_k\left(X_{b_k}-X_{a_k}\right)\right]e^{\diamond \fc_k}.
\end{equation}
Note that the second term vanishes, when $X$ is a Gaussian martingale and the integrand $Z$ is adapted to the filtration generated by $X$. Hence, we recover the It\^o integral for simple integrands 
in this case. The second term can be interpreted as a trace term, which incorporates the memory effects of the driving Gaussian process and/or the nonadaptedness of the integrand. It forces 
the integral to have zero expectation.
In order to extend this integral to a larger class of integrands we first compute its $S$-transform.
We obtain thanks to (\ref{eq:S_transform_first_chaos}) and (\ref{eq:S_transform_properties}) by elementary manipulations
\begin{eqnarray}
\label{eq:simple_integral_S_transform}
\left(S \int_0^T Z_s \ud^\diamond X_s\right)(\hh) &=& 
\sum_{k=1}^n \left(Se^{\diamond \fc_k}\right)(\hh)\left(S\left(X_{b_k}-X_{a_k}\right)\right)(\hh)\nonumber \\
&=& \sum_{k=1}^n \left(Se^{\diamond \fc_k}\right)(\hh)(\inth(b_k)-\inth(a_k)),
\end{eqnarray}
which we can think of as an integral of the deterministic step function $(SZ_t)(\hh)$ with respect to $\inth\in CM$.

Given a function $\cinth \in CM$, we say that a Hilbert space $\left(\mathcal{H},\Vert\cdot\Vert_{\mathcal{H}}\right)$ is a \emph{space of integrands} for the integrator $\cinth$, whenever
\begin{enumerate}
 \item The set $\{\1_{(0,t]}:\quad t\in[0,T]\}$ is total in $\mathcal{H}$;
\item the map $\1_{(0,t]}\mapsto \cinth(t)$ can be extended to a continuous linear functional from $\mathcal{H}$ to $\mathbb{R}$.
\end{enumerate}
In this case we denote the extension as $\int_0^T \,\cdot \;d^\mathcal{H}\cinth$.

\begin{defn}
 We say a Hilbert space $\left(\mathcal{H},\Vert\cdot\Vert_{\mathcal{H}}\right)$ is \emph{appropriate for the extension of the Skorokhod integral with respect to $X$} 
(\emph{appropriate for $X$}, for short), if it is a space of integrands 
for integrators from a dense subset of the Cameron-Martin space associated to $X$.  
\end{defn}

Consequently, if $\mathcal{H}$ is appropriate for $X$, then the subspace
$$
\Cset_{\mathcal{H}} := \left\{\hh \in \HH : \mathcal{H}\textnormal{ is a space of integrands for }\inth\right\}
$$
is dense in $\HH$, and hence every random variable in $L^2_X$ is uniquely determined by the restriction of the $S$-transform to $\Cset_{\mathcal{H}}$.

Given a Hilbert space $\mathcal{H}$, which is appropriate for $X$, we next extend the integral for simple intgrands in (\ref{eq:simple_integral}) to a closed operator from a dense
subset of $L^2(\Omega,\mathcal{H})= L^2_X\otimes \mathcal{H}$ to $L^2_X$ (with the tensor product understood in the sense of Hilbert spaces).
\begin{defn}
Suppose $\mathcal{H}$ is appropriate for $X$. We say that  $Z\in L^2(\Omega,\mathcal{H})$ 
belongs to the domain of the $\mathcal{H}$-Skorokhod integral with respect to $X$, if there is a random variable
$\int_0^T Z \ud^{\diamond_\mathcal{H}}X \in L^2_X$ such that, for every $\hh\in \Cset_{\mathcal{H}}$,
\begin{equation}
\label{eq:abstract_integral_def}
\left(S \int_0^T Z \ud^{\diamond_\mathcal{H}}X\right)(\hh) = \int_0^T \left(SZ\right)(\hh)\ud^\mathcal{H}\inth. 
\end{equation}
In this case $\int_0^T Z \ud^{\diamond_\mathcal{H}}X$ is uniquely determined and called the  \emph{$\mathcal{H}$-Skorokhod integral of $Z$ with respect to $X$}.
\end{defn}
To be precise, $\left(SZ\right)(\hh)$ here denotes (for fixed $\hh\in \Cset_{\mathcal{H}}$) the element in $\mathcal{H}$, which is obtained by applying the $L^2_X$-inner product of
the Wick exponential $e^{\diamond I(\hh)}$ to the $L^2_X$-coordinate of $Z$, i.e.
$$
\left(SZ\right)(\hh)=(\E^{Q_\hh} \otimes \textnormal{id}_{\mathcal{H}})(Z),\quad  \ud Q_\hh = e^{\diamond I(\hh)} \ud P.
$$
Note that, for fixed $\hh\in \Cset_{\mathcal{H}}$, this extended $S$-transform $Z\mapsto \left(SZ\right)(\hh)$ is a continuous operator from $L^2(\Omega,\mathcal{H})$ to $\mathcal{H}$.
\begin{pro}
 Suppose $\mathcal{H}$ is appropriate for $X$. Then, the $\mathcal{H}$-Skorokhod integral with respect to $X$ is a densely defined closed operator from $L^2(\Omega,\mathcal{H})$ to $L^2_X$. 
Moreover, the space 
of simple integrands $\mathcal{E}$ is included in its domain, and the restriction of the $\mathcal{H}$-Skorokhod integral to $\mathcal{E}$ is given by (\ref{eq:simple_integral}).
\end{pro}
\begin{proof}
 Equation (\ref{eq:simple_integral_S_transform}) implies that $\mathcal{E}$ is included in the domain of the $\mathcal{H}$-Skorokhod integral and that the $\mathcal{H}$-Skorokhod integral
for simple integrands is given by (\ref{eq:simple_integral}). As the Wick exponentials are total in $L^2_X$ and the indicator functions $\1_{(0,t]}$, $t\in [0,T]$ are total in $\mathcal{H}$,
we observe that $\mathcal{E}$ is dense in $L^2(\Omega,\mathcal{H})$. It, hence, remains to show the closedness. To this end suppose that $(Z_n)$ converges to $Z$ in $L^2(\Omega,\mathcal{H})$,
$Z_n$ is in the domain of the $\mathcal{H}$-Skorokhod integral and $\int_0^T Z_n \ud^{\diamond_\mathcal{H}}X$ converges to some random variable $\xi$ in $L^2_X$. Then, for every 
$\hh \in \Cset_{\mathcal{H}}$, $(SZ_n)(\hh)$ converges to $(SZ)(\hh)$ in $\mathcal{H}$, which in turn implies
\begin{eqnarray}
\label{eq:S_transform_general_integral_convergence}
 (S\xi)(\hh)&=&\lim_{n\rightarrow \infty} \int_0^T \left(SZ_n\right)(\hh)\ud^\mathcal{H}\inth=\int_0^T \left(SZ\right)(\hh)\ud^\mathcal{H}\inth,
\end{eqnarray}
because $\mathcal{H}$ is a space of integrands for $\inth$, whenever $\hh\in \Cset_{\mathcal{H}}$. Thus, $\xi$ is the $\mathcal{H}$-Skorokhod integral of $Z$ with respect to $X$.
\end{proof}

Before we exemplify several different constructions of appropriate Hilbert spaces $\mathcal{H}$, let us recall that \emph{fractional Brownian motion} 
with Hurst parameter $H\in (0,1)$ is a centered Gaussian process $X$ with covariance function
$$
R(t,s)=\mathbb{E}[X_tX_s]=\frac{1}{2} \left( t^{2H}+s^{2H}-|t-s|^{2H}\right).
$$

\begin{exa}\label{ex:integral}
(i) For each Gaussian process $X$, its space of deterministic integrands $\HH$ is appropriate. Indeed, for every $\inth \in CM_X$, we observe that
$$
\int_0^T \gg \,\ud^\HH \inth:= S(I(\gg))(\hh)
$$
is a continuous mapping from $\HH$ to $\mathbb{R}$ with $\int_0^T \1_{(0,t]} \ud^\HH \inth=\inth(t)$ by (\ref{eq:S_transform_first_chaos2}). Hence, $\Cset_{\HH}=\HH$. One easily observes 
that the defining equation (\ref{eq:abstract_integral_def}) for the $\HH$-Skorokhod integral reduces to the duality relation between Skorokhod integral and Malliavin derivative restricted 
to the Wick exponentials, which are dense within the set of square integrable random variables with square integrable Malliavin derivative, cp. \cite{nu}. Hence, the 
$\HH$-Skorokhod integral coincides with the classical Skorokhod integral.
\\[0.2cm]
(ii) Let the Hilbert space $\mathcal{H}$ be given by $\mathcal{H}=L^2(\ud \mu)=L^2([0,T],\ud \mu)$ for some finite measure $\mu$ on $[0,T]$. Then, 
$L^2(\ud \mu)$ is a space of integrands for  $\cinth \in CM$, if and only if 
$$
\cinth(t)= \int_0^t \dot{\cinth}(s) \mu(\ud s) \textnormal{ for some }\dot{\cinth}\in L^2(\ud\mu) 
$$
In this case we have for 
${\bf g}\in L^2(\ud\mu)$,
$$
\int_0^T \,{\bf g}\, \ud^{L^2(\ud\mu)} \cinth = \int_0^T {\bf g}(s) \dot{\cinth}(s) \mu(\ud s).
$$
Consequently
$$
 \Cset_{L^2(\ud\mu)}=\left\{\hh \in \HH : \inth(t)= \int_0^t \dot{\inth}(s) \mu(ds) \textnormal{ for some }\dot{\inth}\in L^2([0,T],\ud\mu)    \right\}.
$$
In \cite{bender} sufficient conditions for $X$ are provided to ensure that $L^2(\ud\mu)$ is appropriate for $X$ with the choice $\mu((0,t])=Var(X_t)$ 
and in particular it is shown that, with this choice of $\mu$, $L^2(\ud\mu)$ is appropriate for $X$, whenever $X$ is a finite linear combination of independent fractional Brownian motions with different Hurst 
parameters.
\\[0.2cm]
(iii) Suppose $\mathcal{H}$ is a Hilbert space, $\HH_X$ is densely and continuously embedded in $\mathcal{H}$, and $T:\HH_X \subset \mathcal{H} \rightarrow \mathcal{H}$ is a linear map 
such that $\|T\hh\|_\mathcal{H}=\|h\|_\HH$. Denote by $T^*$ the adjoint operator of $T$. If 
the set 
$$A_T=\{\hh\in \HH_X:\; T\hh \textnormal{ belongs to the domain of } T^*\}  $$
is dense in $\HH_X$, then  $\mathcal{H}$ is appropriate for $X$. Indeed, for $\hh\in A_T$,
\begin{eqnarray*}
\left|\sum \alpha_i \inth(t_i)\right|&=& \left| \la \sum \alpha_i\1_{(0,t_i]},\hh \ra_\HH\right|=\left|\la \sum \alpha_i\1_{(0,t_i]}, T^*T \hh \ra_\mathcal{H}\right|\\ &\leq& \|T^*T \hh  \|_\mathcal{H} \|\sum \alpha_i\1_{(0,t_i]} \|_\mathcal{H},
\end{eqnarray*}
which implies that $A_T\subset \Cset_{\mathcal{H}}$. In this situation the $\mathcal{H}$-Skorokhod integral coincides with the extended divergence operator of \cite{nualartleon}.
By \cite{nualartleon}, for fractional Brownian motion with Hurst parameter $H<1/2$ the space $\mathcal{H}=L^2\left([0,T],\ud t\right)$ is covered in the setting of this example.
\end{exa}

In order to study BSDEs, we must be able to define $\mathcal{H}$-Skorokhod integrals over time intervals $[t,T]$ rather than on the whole interval $[0,T]$. To this end we introduce the following additional
condition on an appropriate space $\mathcal{H}$ for $X$:

\begin{description}
\item[\textbf{(I$_{\mathcal{H}}$)}] For every $r\in[0,T]$ there is a continuous linear operator $\I^{\mathcal{H}}_r : \mathcal{H}\mapsto \mathcal{H}$ such 
that $\I^{\mathcal{H}}_r \left(\1_{(0,t]}\right) = \1_{(0,t\wedge r]}$. 
\end{description} 
The assumption \textbf{(I$_{\mathcal{H}}$)} is, e.g., clearly satisfied for the space $\mathcal{H}=L^2([0,T],\ud\mu)$ in Example \ref{ex:integral}, (ii), where the operator $\I^{\mathcal{H}}_r$ is just the 
multiplication operator with 
the indicator function $\1_{(0,r]}$. 
\begin{defn}
Let $0\leq a< b\leq T$ and suppose $\mathcal{H}$ is an appropriate space of integrands satisfying  (I$_{\mathcal{H}}$). The integral $\int_a^b Z \ud^{\diamond_{\mathcal{H}}}X$ is defined as
$$
\int_a^b Z\ud^{\diamond_{\mathcal{H}}}X = \int_0^T \I^{\mathcal{H}}_{(a,b]}(Z)\ud^{\diamond_{\mathcal{H}}}X,
$$
provided the right-side exists, where $\I^{\mathcal{H}}_{(a,b]}=\textnormal{id}_{L^2_X}\otimes(\I^{\mathcal{H}}_b-\I^{\mathcal{H}}_a )$.
\end{defn}

If condition \textbf{(I$_{\mathcal{H}}$)} is satisfied for the space of deterministic integrands $\mathcal{H}=\HH_X$ of $X$, we can,
for every $\hh\in\HH_X$ define an indefinite Wiener integral by
$$
\int_0^t \hh \,\ud X := I( \I^{\HH_X}_t(\hh)),\quad  0\leq t \leq T,
$$
and for fixed $t$, the indefinite Wiener integral is a continuous mapping from $\HH_X$ to the first chaos $H_X$. 
For most of the paper we shall assume that such construction of indefinite Wiener integrals is possible:
\begin{defn}\label{def:indefinite}
We say that $X$ has an \emph{indefinite Wiener integral} if the space of deterministic integrands $\HH_X$ satisfies (I$_{\HH_X}$).
\end{defn}
In Section \ref{sec:indefinite} we discuss how to check this property, and, in particular, show that linear combinations 
of independent fractional Brownian motions have an indefinite Wiener integral.

\section{Linear BSDEs and discussion of main results}\label{sec:BSDE}

We now consider linear BSDEs of the form
\begin{equation}\label{BSDE_differentialform}
 \ud Y_t = \left(a(t)Y_t + G_t\right)\ud\gamma(t) +  Z \ud \cinth(t) +  Z \ud^{\diamond} X_t,\quad Y_T=\xi,
\end{equation}
and assume throughout this section that $\cinth$ belongs to the Cameron-Martin space of $X$, $\gamma$ is a continuous function of bounded variation, $a$ is a measurable function and $G$ is an adapted and measurable process satisfying
$$
\int_0^T  |a(s)| \ud |\gamma|(s)+ \E\left[\left( \int_0^T |G_s| \ud |\gamma|(s)\right)^2\right]  <\infty,
$$
where $|\gamma|$ denotes the total variation of $\gamma$.
Moreover, the terminal condition $\xi$ is supposed to belong to $L^2_X$.
We show that a  solution to such equation must necessarily be given in terms of the quasi-conditional expectation operator, which we introduce below. However, first we define a concept of \emph{mild solution} to the above BSDE.
\begin{defn}\label{def:mild}
We say that a triplet  $(Y,Z,\mathcal{H})$ is a \emph{mild solution} to the BSDE (\ref{BSDE_differentialform}), if
\begin{enumerate}
 \item 
$\mathcal{H}$ is an appropriate Hilbert space for $X$ satisfying (I$_{\mathcal{H}}$) and it is a space of integrands for $\cinth$; 
\item $Y$ is an adapted and measurable process satisfying $Y_t\in L^2_X$ for every $t\in [0,T]$ and 
$
\E\left[\left( \int_0^T |a(s) Y_s| \ud |\gamma|(s)\right)^2\right] < \infty;
$
\item $Z \in L^2\left(\Omega,\mathcal{H}\right)$
such that the $\mathcal{H}$-Skorokhod integrals $\int_t^T Z \ud^{\diamond_{\mathcal{H}}}X$ exists for every $t\in[0,T]$;
\item $(Y,Z)$ satisfies the integral form of (\ref{BSDE_differentialform}) with respect to the appropriate space $\mathcal{H}$, i.e.
\begin{equation}
\label{eq:BSDE_abstract}
Y_t = \xi - \int_t^T \left(a(s)Y_s + G_s\right)\ud\gamma(s) - \int_t^T Z \ud^{\mathcal{H}} \cinth - \int_t^T Z \ud^{\diamond_{\mathcal{H}}}X,
\end{equation}
holds for every $t\in[0,T]$ in $L^2_X$.
\end{enumerate}
Here, $\int_t^T Z \ud^{\mathcal{H}} \cinth =\int_0^T  \I^{\mathcal{H}}_{(t,T]}(Z) \ud^{\mathcal{H}} \cinth$ and the integral with respect to $\cinth$ is applied to the $\mathcal{H}$-coordinate 
of an element in $L^2(\Omega,\mathcal{H})= L^2_X\otimes \mathcal{H}$.
\end{defn}

\begin{rem}
 Compared to the usual setting of BSDEs driven by semimartingales, the above notion of a solution is mild in at least two respects. First of all one is free to choose among several extensions 
of the Skorokhod-integral by choosing the appropriate space $\mathcal{H}$. Additionally, we don't ask the $Z$-part of the solution to be adapted. Indeed, in general $Z$  (depending on the choice 
of $\mathcal{H}$) need not even be a stochastic process in time.
\end{rem}

The next example illustrates why the flexibility in choosing the appropriate space $\mathcal{H}$ is important.
\begin{exa} \label{exa:BSDE}
 Suppose $X$ is a fractional Brownian motion with Hurst parameter $H< 1/4$. We consider the three choices of appropriate spaces (in view of Example \ref{ex:integral}, (i)--(iii))
$$
\HH_X \subset L^2\left(\ud V\right)\subset L^2\left(\ud t\right),
$$
where $V(t)=t^{2H}$ is the variance function of $X$. The inclusions hold in the sense of continuous embeddings, with the first
one following from a weighted Hardy-Littlewood 
inequality for fractional 
Riemann-Liouville integrals
and the characterization of the space of deterministic integrands of a fractional Brownian motion with Hurst parameter $H<1/2$ in Pipiras and Taqqu \cite{piptaq}. More precisely,
we have $\ff \in \HH_X$ if and only if $\ff = K^*g$ for some $g\in L^2\left(\ud t\right)$, where the operator $K^*$ is given by 
$\left(K^*g\right)(t) = c_H t^{\frac12 -H}\left(I_{T-}^{\frac12 - H} \cdot^{H-\frac12}g(\cdot)\right)(t)$ and 
$I^{1/2-H}_{T-}$ denotes the Riemann-Liouville fractional integral. In this case, $\| \ff\|_{\HH_X}=\|g\|_{L^2(dt)}$.
Thus, the continuous embedding of 
$\HH_X$ into $L^2\left(\ud V\right)$ follows from equation $(5.45')$ in Samko et al. \cite{s-k-m}.
Hence, the three $\mathcal{H}$-Skorokhod
integrals are extensions of each other, with the space $L^2\left(\ud t\right)$ leading to the largest set of integrands among these three choices, and the space of deterministic integrands to the smallest one.

We consider the linear BSDE
$$
dY_t=2Ht^{2H-1}\ud t +    Z \ud \cinth(t)+ Z \ud^{\diamond} X_t,\quad Y_T=(X_T+V(T))^2
$$
for $\cinth(t)=V(t)=t^{2H}$. This function $\cinth$ indeed belongs to the Cameron-Martin space of a fractional Brownian motion with Hurst parameter $H<1/2$ by (2.44) in Samko et al. 
\cite{s-k-m} and the 
characterization 
of the Cameron-Martin space of a fractional Brownian motion in Decreusefond and \"Ust\"unel \cite[Remark 3.1]{dec-ust}. 

With the choice $\mathcal{H}=L^2\left(\ud V\right)$ we can apply the It\^o formula in the form of Theorem 3.3 in \cite{bender} to $(X_t+V(t))^2$ in order to derive
\begin{eqnarray*}
 (X_t-V(t))^2&=&(X_T-V(T))^2- \int_t^T 1 \ud V(s)- \int_t^T 2(X_s+V(s)) \ud V(s) \\ && - \int_t^T 2(X_s+V(s)) \ud^{\diamond_{L^2(dV)}}X.
\end{eqnarray*}
By Example \ref{ex:integral}, (ii), $L^2\left(\ud V\right)$ is a space of integrands for $\cinth$ and 
$$
\int_t^T 2(X_s+V(s)) \ud V(s) =\int_t^T 2(X_s+V(s)) \ud^{L^2(\ud V)} \cinth. 
$$ 
Hence, with $Y_t=(X_t+V(t))^2$ and $Z_t=2(X_t+V(t))$, the triplet $(Y,Z, L^2\left(\ud V\right))$ is a mild solution to the above linear BSDE. 

We claim that neither $(Y,Z, L^2\left(\ud t\right))$ nor  $(Y,Z, \HH_X)$ is a mild solution.  The space $\mathcal{H}=L^2\left(\ud t\right)$ is too large. Writing 
$$
\cinth(t)=\int_0^t \dot \cinth(s) \ud s\quad \textnormal{ for } \quad \dot \cinth(s)=2Hs^{2H-1},
$$
we observe that $ \dot \cinth\notin L^2\left(\ud t\right)$ for $H< 1/4$, and thus $L^2(dt)$ is not a space of integrands for $\cinth$ by Example \ref{ex:integral}, (ii). 
Similarly, the space $\HH_X$ is too small. Indeed, we have $t^{2H} \in \HH_X$ (which proof is postponed to the Appendix \ref{sec:A}), and 
consequently $Z_t = 2(X_t + V(t)) \notin L^2\left(\Omega,\HH_X\right)$ since it is proved in \cite{cheriditonualart} that $X_t \notin L^2\left(\Omega,\HH_X\right)$ for $H\leq \frac14$.
\end{exa}

We next present our main result on non-existence for linear Gaussian BSDEs driven by non-martingales.

\begin{thm}\label{thm:nonexistence}
Suppose $X$ has an indefinite Wiener integral. If $X$ is not a martingale, then, for any choice of the coefficients $a$, $\gamma$, $\cinth$, and $G$ (satisfying the standing assumptions of this section), there exists a terminal value $\xi \in L^2_X$ such that 
linear BSDE (\ref{BSDE_differentialform}) does not have a mild solution.
\end{thm}
\begin{rem}\label{rem:BSDE_martingale}
We note the following partial converse of Theorem \ref{thm:nonexistence}. Suppose that $X$ is a martingale, and, additionally, that its variance function $V(t)=\E[X_t^2]$ is continuous and strictly 
increasing with inverse function $U(t)$. Then, $W_t=X_{U(t)}$ is a Brownian motion, and, hence, $X=W_V$ has the martingale representation property 
with respect to its natural filtration. Applying, Theorem 6.1 in \cite{karoui-huang} one observes:
For every $\cinth\in CM_X$, $a\in L^1(\ud V)$, $\xi\in L^2_X$, and adapted process $G$ satisfying $\E[\int_0^T G_s^2 \ud V(s)]<\infty$, there
exists a pair of adapted processes $(Y,Z)$ such that $(Y,Z,\HH_X)$ is a mild solution to
$$
\ud Y_t = \left(a(t)Y_t + G_t\right)\ud V(t) +  Z \ud \cinth(t) +  Z \ud^{\diamond} X_t,\quad Y_T=\xi.
$$
Here, we use that $\HH_X=L^2(\ud V)$ and $CM_X=\{ \int_0^\cdot \dot \cinth (s) \ud V(s);\; \dot \cinth \in L^2(\ud V)\}$ in the martingale case. Consequently,
$$
\int_t^T  Z \ud^{\HH_X} \cinth=\int_t^T Z_s \dot \cinth (s) \ud V(s)
$$
and the restriction of the Skorokhod integral to adapted integrands coincides with the classical It\^o stochastic integral for continuous Gaussian martingales.
\end{rem}

The proof of Theorem \ref{thm:nonexistence} relies on the concept of shifted quasi-conditional expectation, which we introduce now.
To this end, suppose that $X$ has an indefinite Wiener integral.
We denote the adjoint operator of $\I^{\mathcal{\HH}}_r$ by $(\I^{\mathcal{\HH}}_r)^*$ and abbreviate
$ \hh^r:= (\I^{\mathcal{\HH}}_r)^*(\hh)$ for $\hh\in \HH$. We also drop the superscripts $\HH$ and write simply $\I_r$ and $\I_r^*$ instead of $\I^{\HH}_r$ and $(\I^{\HH}_r)^*$ if no confusion 
can arise.

\begin{defn}[$\cc$-shifted quasi-conditional expectation]
Suppose $X$ has an indefinite Wiener integral. Let $r\in[0,T]$ and $\cc\in \HH$ be fixed. We define the subset $D_r^{\cc} \subset L^2_X$ by
$$
D_r^{\cc} = \left\{\xi \in L^2_X : \exists \eta^{\cc}_r\in L^2_X\text{ s.t. }\left(S\xi\right)((\hh+\cc)^r-\cc) = \left(S\eta^{\cc}_r\right)(\hh),\forall \hh \in \HH\right\}.
$$
For $\xi \in D_r^{\cc}$, the corresponding $\eta^{\cc}_r \in L^2_X$ is uniquely determined. Such $\eta^{\cc}_r$ is called \emph{$\cc$-shifted quasi-conditional expectation} 
of $\xi$ and denoted by $\widetilde{E}^{\cc}\left[\xi | \mathcal{F}_r^X\right]$. In the case $\cc=0$ we write $\widetilde{E}\left[\xi|\mathcal{F}_r^X\right]$ and $D_r$ instead 
of $\widetilde{E}^{0}\left[\xi | \mathcal{F}_r^X\right]$ and $D_r^0$, and we call the corresponding operator the \emph{quasi-conditional expectation}.
\end{defn}

Note that, by Theorem \ref{thm:martingale_change_of_measure} below, shifted quasi-conditional expectation coincides with the 
classical conditional expectation under a Cameron-Martin shift in the martingale case.

We will study shifted quasi-conditional expectations in more detail in Section \ref{sec:QC}, to which we also postpone the proofs of the following theorems.
The first one reveals the fact that the domain $D_r^{\cc}$ can be the whole $L^2_X$ only in the case of martingales, and the second one 
is a uniqueness and representation theorem for the $Y$-part of a mild solution in terms of operator $\widetilde{E}^{\cc}\left[\cdot |\mathcal{F}_r^X\right]$.
\begin{thm}
\label{thm:mg_domain_connection}
Suppose $X$ has an indefinite Wiener integral. Then the following are equivalent: 
\begin{enumerate}
\item $X$ is not a martingale,
\item for every $c \in \HH$ there is an $r\in[0,T]$ such that $D_r^{\cc} \subsetneqq L^2_X$.
\end{enumerate}
\end{thm}
\begin{thm}
\label{thm:solution_Y}
Under the standing assumptions of this section on the coefficients, suppose that $(Y,Z,\mathcal{H})$ is a mild solution to the linear BSDE (\ref{BSDE_differentialform}). 
Define 
$$
\tilde{\xi} = \xi - \int_0^T A(s)G_s\ud\gamma(s),\quad A(t) = e^{\int_t^T a(s) \ud \gamma(s)}.
$$
Then, $\tilde\xi \in D^{\cc}_t$ for every $t\in[0,T]$, and
\begin{equation}
\label{eq:solution_Y_abstract}
Y_t = A(t)^{-1}\widetilde{E}^{\cc}\left[\tilde\xi|\mathcal{F}_t^X\right]+ A(t)^{-1}\int_0^t A(s)G_s \ud \gamma(s).
\end{equation}
\end{thm}
We note that the non-existence result in Theorem \ref{thm:nonexistence} is a direct consequence of Theorems \ref{thm:mg_domain_connection} and
\ref{thm:solution_Y}.

\section{On quasi-conditional expectation}
\label{sec:QC}
In this section, we study in detail the shifted quasi-conditional expectation operator 
$\widetilde{E}^{\cc}\left[\cdot | \mathcal{F}_r^X\right]$.

We first explore how quasi-conditional expectation acts on the first chaos, and, at the same time, 
we provide an equivalent characterization for the existence of indefinite Wiener integrals 
in terms of complementary spaces. To this end, we introduce some notation.
Since $X_0=0$, the first chaos $H_X$ can equivalently be defined as the closed linear subspace spanned by increments
$X_u - X_s, \quad u,s\in[0,T]$. For a fixed number $r\in [0,T]$, denote by $\vr$ and $\vT$ the closed linear spaces 
$$
\vr = \overline{\text{span}\{X_u : u\in[0,r]\}}
$$
and
$$
\vT = \overline{\text{span}\{X_u-X_r : u\in[r,T]\}}
$$
respectively, and denote the orthogonal complements of these spaces by $\vro$ and $\vTo$.

\begin{thm}
\label{thm:quasi_equivalent}
The following statements are equivalent:
\begin{enumerate}
\item $X$ has an indefinite Wiener integral,
\item the spaces $\vr$ and $\vT$ are complementary, i.e. every $\fc \in H_X$ can be uniquely decomposed as
\begin{equation}
\label{eq:decomposition}
\fc = \fc_{\vr} + \fc_{\vT},
\end{equation}
where $\fc_{\vr} \in \vr$ and $\fc_{\vT} \in \vT$,
\item the spaces $\vro$ and $\vTo$ are complementary.
\end{enumerate}
In this case, $H_X \subset D_r$ for every $r\in[0,T]$, and for every element $h=I(\hh)$ we have
\begin{equation}
\label{eq:quasi_first_chaos}
\widetilde{E}\left[\fc | \mathcal{F}_r^X\right] = h_{\vr} = I\left(\I_r\hh\right).
\end{equation}
\end{thm}
\begin{proof}
The equivalence of \textbf{(ii)} and \textbf{(iii)} is well-known and a simple exercise. Indeed, if, e.g., $\vr$ and $\vT$ are complementary and $\fc \in H_X$ has a decomposition 
$\fc = \fc_{\vr} + \fc_{\vT}$, then the oblique projection $\Phi\fc = \fc_{\vr}$ is a 
bounded linear operator (see, e.g. \cite[pp. 97, Theorem 13.2]{conway}), and hence the fact that $\vro$ and $\vTo$ are complementary follows by considering the adjoint operator $\Phi^*$ of $\Phi$. 
It remains to prove equivalence of \textbf{(i)} and \textbf{(ii)}. \\
\textbf{(ii)} $\Rightarrow$ \textbf{(i)}: Suppose $\vr$ and $\vT$ are complementary and define
a map $\widetilde{E}_r : H_X \mapsto H_X$ by $\widetilde{E}_r \fc = \fc_{\vr}$. This operator
satisfies $\widetilde{E}_r X_t = X_{t\wedge r}$. Moreover, since $\vr$ and $\vT$ are complementary, the operator $\widetilde{E}_r$ is bounded (as already noted above). From this the existence of an indefinite Wiener integral follows from the isometry between $H_X$ and $\HH_X$.\\
\textbf{(i)} $\Rightarrow$ \textbf{(ii)}: Suppose $X$ has an indefinite Wiener integral, and, hence, quasi-conditional expectation is defined. Take $f=I(\ff)\in H_X$.
Then, for $\hh\in \HH$, by (\ref{eq:S_transform_first_chaos}) and the isometry between first chaos and the space of deterministic integrands,
$$
(Sf)(\hh^r)=\langle \ff, \I_r^*(\hh)\rangle_{\HH}=\langle \I_r\ff,\hh \rangle_{\HH}= S(I(\I_r\ff))(\hh).
$$
Thus, the first chaos is then contained in the domain of quasi-conditional expectation and
$$
\widetilde E[f|\mathcal{F}^X_r]=I(\I_r\ff).
$$
We next take $f_{\vr}=I(\ff_{\vr})\in \vr$ and consider a sequence $f_n=I(\ff_n)$ with $\ff_n=\sum_i \alpha^{(n)}_i {\bf 1}_{(0,t_i^{(n)}]}$, $0\leq 
t_i^{(n)}\leq r,$ converging to $f_{\vr}$ in $H_X$. Then, the same calculation as above shows, for $\hh\in \HH$,
$$
(Sf_{\vr})(\hh^r)=\lim_{n\rightarrow \infty} \langle \I_r\ff_n,\hh \rangle_{\HH}=\lim_{n\rightarrow \infty} \langle \ff_n,\hh \rangle_{\HH}= S(I(\ff_{\vr}))(\hh),
$$
because $\I_r\ff_n=\ff_n$. Hence,
$$
\widetilde E[f_{\vr}|\mathcal{F}^X_r]=f_{\vr}
$$
for $f_{\vr}\in \vr$. Finally, if $f_{\vT}=I(\ff_{\vT})\in \vT$ and  the sequence $f_n=I(\ff_n)$ with $\ff_n=\sum_i \alpha^{(n)}_i {\bf 1}_{(r,t_i^{(n)}]}$, $r\leq 
t_i^{(n)}\leq T,$ converges to $f_{\vT}$ in $H_X$, we obtain, for $\hh\in \HH$,
$$
(Sf_{\vT})(\hh^r)=\lim_{n\rightarrow \infty} \langle \I_r\ff_n,\hh \rangle_{\HH}=\lim_{n\rightarrow \infty} \langle \I_r\I_{(r,T]}\ff_n,\hh \rangle_{\HH}= 0,
$$
and, thus,
$$
\widetilde E[f_{\vT}|\mathcal{F}^X_r]=0
$$
for $f_{\vT}\in \vT$. With these computations we are ready to prove that $\vr$ and $\vT$ are complementary. That is, each $\fc$ admits a unique decomposition of form \eqref{eq:decomposition}. To show {\it uniqueness of the decomposition}, suppose that $f=f_{\vr}+f_{\vT}$ for some 
$f_{\vr}\in \vr$ and $f_{\vT}\in \vT$. Then, by the above computations,
$$
I(\I_r\ff)=\widetilde E[f|\mathcal{F}^X_r]=\widetilde E[f_{\vr}|\mathcal{F}^X_r]+\widetilde E[f_{\vT}|\mathcal{F}^X_r]=f_{\vr}.
$$
This shows uniqueness and (\ref{eq:quasi_first_chaos}). It remains to show {\it existence of the decomposition}. 
We first assume that $f=\sum_i \alpha_i X_{t_i}=I(\sum_i \alpha_i {\bf 1}_{(0,t_i]})=I(\ff)$ is a simple element in the first chaos. Then, $f$ admits 
the decomposition $f=f_{\vr} + f_{\vT}$ with
\begin{eqnarray*}
 f_{\vr}&=&\sum_i \alpha_i X_{t_i\wedge r}= I(\I_r\ff) \\
 f_{\vT}&=& \sum_i \alpha_i (X_{t_i\vee r}-X_r)= \sum_i \alpha_i (X_{t_i}-X_{t_i\wedge r})= I(\I_{(r,T]}\ff).
\end{eqnarray*}
Approximating a general $f\in H_X$ by a sequence of such simple elements $(f_n)$ and applying the continuity of $\I_r$ and $\I_{(r,T]}$ implies 
the decomposition $f=I(\I_r\ff)+I(\I_{(r,T]}\ff)$ for general $f\in H_X$, where (by the approximation argument)  
$I(\I_r\ff)\in \vr$ and $I(\I_{(r,T]}\ff)\in \vT$. 
\end{proof}

\begin{rem}
\label{rem:quasi_indefinite}
The above proof shows that the norms of the operators $\I_r: \HH_X \mapsto \HH_X$ and quasi-conditional expectation restricted to the first chaos $\widetilde{E}\left[\cdot | \mathcal{F}_r^X\right]: H_X \mapsto H_X$ are equal, i.e.
$$
\sup_{h \in H_X, h\neq 0} \frac{\Vert\widetilde{E}\left[h | \mathcal{F}_r^X\right]\Vert_{L^2_X}}{\Vert h\Vert_{L^2_X}} = \sup_{\hh \in \HH, \hh \neq 0} \frac{\Vert \I_r \hh\Vert_{\HH}}{\Vert \hh\Vert_{\HH}} =: \Vert \I_r \Vert_{op}.
$$ 
We also note that  application of the adjoint operator $\I_r^*$ of $\I_r$ to  $\hh\in \HH$ corresponds to the element 
$\fc_{\vTo}$ in the decomposition $\fc = \fc_{\vro} + \fc_{\vTo}$, i.e., for $\fc = I(\hh)$ we 
have $\fc_{\vTo} = I\left(\I_r^*\hh\right)$. Indeed, for every $I(\ff)\in L_{[r,T]}$, we have
$$
\E[I\left(\I_r^*\hh\right)I(\ff)]=\langle \I_r^*\hh,\ff \rangle_{\HH}= \langle \hh,\I_r \I_{[r,T]}\ff \rangle_{\HH}=0,
$$
and, for every  $I(\ff)\in L_{[0,r]}$, we have,
$$
\E[(I\left(\hh\right)-I\left(\I_r^*\hh\right))I(\ff)]= \langle \hh,\I_r \ff \rangle_{\HH}-\langle \hh,\I_r \I_{r}\ff \rangle_{\HH}=0.
$$
\end{rem}

The following elementary example shows how quasi-conditional expectation ''behaves'' like classical conditional expectation.
\begin{exa}
\label{ex:simple}
 For every $\cc \in \HH$ and any $f \in H_X$ we have $f \in D_r^{\cc}$ and $e^{\diamond f} \in D_r^{\cc}$. Moreover,
\begin{eqnarray}
 \label{eq:quasi_first_chaos_shift}\widetilde{E}^{\cc}\left[f | \mathcal{F}_r^X\right]&=& \widetilde{E}\left[f | \mathcal{F}_r^X\right] 
- \E\left[\left(f - \widetilde{E}\left[f | \mathcal{F}_r^X\right]\right)I(\cc)\right] \\
 \widetilde{E}^{\cc}\left[e^{\diamond \,f}| \mathcal{F}_r^X\right]&=&  e^{-\E\left[\left(f - \widetilde{E}\left[f | \mathcal{F}_r^X\right]\right) I(\cc)\right]} e^{\diamond \,\widetilde{E}\left[f | \mathcal{F}_r^X\right]}.
\end{eqnarray}
In particular, for every $t\in[0,T]$ we have, thanks to (\ref{eq:quasi_first_chaos}),
\begin{eqnarray*}
 \widetilde{E}^{\cc}\left[X_t | \mathcal{F}_r^X\right]&=& X_{t\wedge r} - \E\left[\left(X_t - X_{t\wedge r}\right)
 I(\cc)\right] \\
 \widetilde{E}^{\cc}\left[e^{\diamond \,X_t}| \mathcal{F}_r^X\right]&=&  e^{-\E[(X_t-X_{t\wedge r}) I(\cc)]} e^{\diamond \,X_{t\wedge r}}.
\end{eqnarray*}
Indeed, for $f = I(\ff)$ and any $ \hh\in \HH$, we have by (\ref{eq:quasi_first_chaos}),
\begin{eqnarray*}
 && (Sf)((\hh+\cc)^r-\cc)\\ &=& \langle \ff, (\I^{\mathcal{\HH}}_r)^*(\hh+\cc)-\cc \rangle_\HH= \langle \I^{\mathcal{\HH}}_r \ff, \hh+\cc\rangle_\HH-\langle \ff,\cc\rangle_\HH 
\\ &=& \langle \I^{\HH}_r \ff, \hh \rangle_\HH - \langle \ff - \I^{\HH}_r \ff,\cc \rangle_\HH = \E\left[\widetilde{E}_r \left[f | \mathcal{F}_r^X\right] I(\hh)\right]-\E\left[\left(f - \widetilde{E}\left[f| \mathcal{F}_r^X\right]\right)I(\cc)\right] 
\\ &=& \left(S\left(\widetilde{E}\left[f| \mathcal{F}_r^X\right] - \E\left[\left(f - \widetilde{E}\left[f| \mathcal{F}_r^X\right]\right)I(\cc)\right]\right)\right)(\hh).
\end{eqnarray*}
Similarly, by (\ref{eq:S_transform_properties}),
\begin{eqnarray*}
 && \left(Se^{\diamond \,f}\right)((\hh+\cc)^r-\cc)\\ &=& e^{\langle \ff, (\I^{\mathcal{\HH}}_r)^*(\hh+\cc)-\cc \rangle_\HH}=e^{\langle \I^{\HH}_r \ff, \hh \rangle_\HH - \langle \ff - \I^{\HH}_r \ff,\cc \rangle_\HH  } \\
&=& \left(S\left((e^{-\E\left[\left(f - \widetilde{E}\left[f | \mathcal{F}_r^X\right]\right) I(\cc)\right]} e^{\diamond \,\widetilde{E}\left[f | \mathcal{F}_r^X\right]}\right)\right)(\hh).
\end{eqnarray*}
Thus, for $\cc=0$, the above formulas mimic the martingale properties of a Brownian motion $W$ and its stochastic exponential for the (in general) non-martingale $X$ and its Wick exponential. For general 
$\cc$, the above formulas generalize corresponding very well-known results for the conditional expectation of a Brownian motion and its stochastic exponential under a change of measure induced by a 
Cameron-Martin shift, cp. also Theorem \ref{thm:martingale_change_of_measure} below.
\end{exa}
As an immediate consequence of this example we obtain:
\begin{thm}
\label{thm:quasi}
Suppose $X$ has an indefinite Wiener integral and let $c \in \HH$ be given. Then the mapping $\widetilde{E}^{\cc}[\xi |\mathcal{F}_r^X] : D_r^{\cc} \mapsto L^2_X$ is a densely defined linear closed operator. 
\end{thm}
\begin{proof}
The linearity and closedness are obvious by the definition. Furthermore, $\widetilde{E}^{\cc}\left[\cdot |\mathcal{F}_r^X\right]$ is densely defined since 
$\text{span}\left(e^{\diamond \fc}, \fc \in H_X\right)$ is dense in $L^2_X$, and $e^{\diamond \fc} \in D_r^{\cc}$ for every $ \fc\in H_X$ by Example \ref{ex:simple}.
\end{proof}

We next relate shifted quasi-conditional expectation to classical conditional expectation in the martingale case.

\begin{thm}
\label{thm:martingale_change_of_measure}
Suppose $X$ is a martingale. Then $X$ has an indefinite Wiener integral. Moreover, for any 
$\xi \in L^2_X$ and any $\cc \in \HH_X$ we have $\xi \in D_r^{\cc}$, and 
$$
\widetilde{E}^{\cc}\left[\xi | \mathcal{F}_r^X\right] = \E^{Q_{-\cc}}\left[\xi | \mathcal{F}_r^X\right],
$$ where $\E^{Q_{-\cc}}\left[\cdot | \mathcal{F}_r^X\right]$ denotes the conditional expectation under the change of measure 
$
\ud Q_{-\cc} = e^{\diamond I(-\cc)}\ud \P.
$
\end{thm}
\begin{proof}
Since $X$ is a martingale, the space of deterministic integrands $\HH_X$ is given by $L^2(\ud V)$, where $V(t)$ is the variance of $X_t$. 
This shows that $X$ has an indefinite Wiener integral and that $\I_r=\I_r^*$ is just the multiplication operator with the indicator 
function ${\bf 1}_{(0,r]}$. As, by (\ref{eq:S_transform_properties}), for $\hh \in L^2(dV)$, $0\leq s_1\leq \ldots s_n\leq s\leq t\leq T$,
and $\alpha_1,\ldots,\alpha_n\in \mathbb{R}$,
\begin{eqnarray*}
&&\E\left[e^{\diamond I\left(\hh\textbf{1}_{(0,t]}\right)}e^{\diamond \sum_i \alpha_i X_{s_i}}  \right]=e^{\int_0^T \hh(u)\textbf{1}_{(0,t]}(u)
\left(\sum_i \alpha_i {\bf 1}_{(0,s_i]}(u)\right)dV(u)}\\ &=& \E\left[e^{\diamond I\left(\hh\textbf{1}_{(0,s]}\right)}e^{\diamond \sum_i \alpha_i X_{s_i}}  \right],
\end{eqnarray*}
we observe that $t\mapsto e^{\diamond I\left(\hh\textbf{1}_{(0,t]}\right)}$ is a martingale. Hence, in view of Bayes' formula,
$$
\E^{Q_{-\cc}} \left[\xi | \mathcal{F}_r^X\right] = \E\left[\xi e^{\diamond I\left(-\cc\textbf{1}_{(r,T]}\right)} | \mathcal{F}_r^X\right].
$$
Moreover, by applying H\"older inequality for conditional expectation we get
\begin{equation*}
\left\vert\E^{Q_{-\cc}} \left[\xi | \mathcal{F}_r^X\right]\right\vert^2 \leq \E\left[\xi^2|\mathcal{F}_r^X\right] \E\left[\left\vert e^{\diamond I\left(-\cc\textbf{1}_{(r,T]}\right)}\right\vert^2 | \mathcal{F}_r^X\right].
\end{equation*}
By \eqref{eq:S_transform_properties} we have
$$
\left\vert e^{\diamond I\left(-\cc\textbf{1}_{(r,T]}\right)}\right\vert^2 = e^{\diamond I\left(-2\cc\textbf{1}_{(r,T]}\right)}e^{\int_r^T \cc(u)^2dV(u)},
$$
and thus, together with the above computations, we obtain
$$
\E\left[\left\vert e^{\diamond I\left(-\cc\textbf{1}_{(r,T]}\right)}\right\vert^2 | \mathcal{F}_r^X\right] = e^{\int_r^T \cc(u)^2dV(u)}.
$$
This gives
$$
\E\left[\E^{Q_{-\cc}} \left[\xi | \mathcal{F}_r^X\right]\right]^2 \leq e^{\int_r^T \cc(u)^2dV(u)}\E [\xi^2],
$$
and hence $\E^{Q_{-\cc}}\left[ \xi | \mathcal{F}_r^X\right] \in L^2_X$. 
We, thus, observe by \eqref{eq:S_transform_properties} and the martingale property of $e^{\diamond I\left(\hh\textbf{1}_{(0,t]}\right)}$,
\begin{eqnarray*}
&&\left(S\xi\right)((\hh+\cc)^r - \cc) = \E\left[\xi e^{\diamond \left(\hh\textbf{1}_{(0,r]} + \cc\textbf{1}_{(0,r]}-\cc\textbf{1}_{(0,T]}\right)}\right]\\
&=& \E\left[\xi e^{\diamond I\left(\hh\textbf{1}_{(0,r]} - \cc\textbf{1}_{(r,T]}\right)}\right]=\E\left[\xi e^{\diamond I\left(\hh\textbf{1}_{(0,r]}\right)}e^{\diamond I\left(-\cc\textbf{1}_{(r,T]}\right)}\right]
\\
&=&\E\left[\E \left[\xi e^{\diamond I\left(-\cc\textbf{1}_{(r,T]}\right)} | \mathcal{F}_r^X\right]e^{\diamond I\left(\hh\textbf{1}_{(0,r]}\right)}\right]\\
&=& \E\left[\E^{Q_{-\cc}} \left[\xi  | \mathcal{F}_r^X\right]e^{\diamond I\left(\hh\textbf{1}_{(0,r]}\right)}\right]=\E\left[\E^{Q_{-\cc}}\left[ \xi | \mathcal{F}_r^X\right] e^{\diamond I(\hh)}\right]
\\ &=& \left(S\E^{Q_{-\cc}}\left[ \xi | \mathcal{F}_r^X\right]\right)(\hh).
\end{eqnarray*}
Since $\E^{Q_{-\cc}}\left[ \xi | \mathcal{F}_r^X\right] \in L^2_X$, this shows that $\xi \in D_r^{\cc}$ and 
$\widetilde{E}^{\cc}\left[\xi | \mathcal{F}_r^X\right] = \E^{Q_{-\cc}}\left[\xi | \mathcal{F}_r^X\right]$. 
\end{proof}

We continue by showing that some well-known properties of classical conditional expectation carry over to shifted quasi-conditional 
expectation. We start with the towering property.

\begin{pro}[Towering property]
Suppose $X$ has an indefinite Wiener integral, let $r_1 < r_2$ and $\xi \in D^{\cc}_{r_2}$. 
Then,  $\widetilde{E}^{\cc}\left[\xi |\mathcal{F}_{r_2}^X\right] \in D^{\cc}_{r_1}$,  if and only if $\xi \in D^{\cc}_{r_1}$. In this case,
\begin{equation}
\label{eq:towering}
\widetilde{E}^{\cc}\left[\widetilde{E}^{\cc}\left[\xi |\mathcal{F}_{r_2}^X\right] |\mathcal{F}_{r_1}^X\right] = \widetilde{E}^{\cc}\left[\xi |\mathcal{F}_{r_1}^X\right]. 
\end{equation}
\end{pro}
\begin{proof}
Note first that $\I_{r_1}\circ \I_{r_2}=\I_{r_1}$, which certainly is true for indicator functions and then extends by continuity. Then,
by duality, $\I^*_{r_2}\circ \I^*_{r_1}=\I^*_{r_1}$. Thus, 
\begin{eqnarray*}
 \\ && S\left( \widetilde{E}^{\cc}\left[\xi |\mathcal{F}_{r_2}^X\right]\right)\left((\hh+\cc)^{r_1} - \cc\right)=
 S(\xi)\left(\left((\hh+\cc)^{r_1}\right)^{r_2} - \cc\right)\\ &=& S(\xi)\left((\hh+\cc)^{r_1} - \cc\right),
\end{eqnarray*}
which proves the claim, taking the definition of shifted quasi-conditional expectation into account.
\end{proof}

We next turn to measurability properties of shifted quasi-conditional expectation.

\begin{thm}
\label{lma:measurable}
If $\xi$ is $\mathcal{F}_r^X$-measurable, then $\xi \in D_r^{\cc}$ and $\widetilde{E}^{\cc}\left[\xi |\mathcal{F}_r^X\right] = \xi$. Conversely, let $\xi \in D_r^{\cc}$ such that $\widetilde{E}^{\cc}\left[\xi | \mathcal{F}_r^X\right] = \xi$. Then $\xi$ is $\mathcal{F}_r^X$-measurable.
\end{thm}
\begin{proof}
Suppose first that $\xi$ is $\mathcal{F}_r^X$-measurable. By Example \ref{ex:simple}
and (\ref{eq:quasi_first_chaos}) we observe that $\widetilde{E}^{\cc}\left[e^{\diamond \fc}| \mathcal{F}_r^X\right] = e^{\diamond \fc}$ 
for any $\fc$ of the form  $\fc = \sum_{k=1}^n \alpha_k X_{t_k}$ with $t_k\leq r$ for every $k$.
 Since $\xi$ is $\mathcal{F}_r^X$-measurable, there exists a sequence 
$\xi^{(n)} = \sum_{k=1}^n \alpha_k^{(n)}e^{\diamond \fc_k^{(n)}}$, where
$\fc_k^{(n)} \in \text{span}\left(X_t:t\in[0,r]\right)$, 
such that $\xi^{(n)} \to \xi$ in $L^2_X$. Hence for every $\hh\in \HH$ we get
\begin{eqnarray*}
\left(S\xi\right)((\hh+\cc)^r -\cc) &=& \lim_{n\to \infty} \left(S\xi^{(n)}\right)((\hh+\cc)^r -\cc)
=\lim_{n\to\infty} \left(S\xi^{(n)}\right)(\hh)\\
&=&\left(S\xi\right)(\hh).
\end{eqnarray*}
Conversely, suppose $\xi \in D_r^{\cc}$ such that $\widetilde{E}^{\cc}_r\left[\xi | \mathcal{F}_r^X\right] = \xi$. Then for any $\hh\in \HH$ we have
\begin{equation}
\label{eq:S_equal_case}
\left(S\xi\right)\left((\hh+\cc)^r - \cc\right) = \left(S\xi\right)(\hh).
\end{equation}
Let now $\hh\in \HH$ be such that $I(\hh) \in \vro$. 
By Remark \ref{rem:quasi_indefinite} we have $\hh^r = 0$. Hence we get 
$$
\left(S\xi\right)(\hh) = \left(S\xi\right)\left(\cc^r - \cc\right).
$$
Furthermore, by decomposing $\xi$ as 
$$
\xi = \E\left[\xi | \mathcal{F}_r^X\right] + \left[\xi -  \E\left[\xi | \mathcal{F}_r^X\right]\right]
$$
and using the fact $e^{\diamond I(\hh)}$ belongs to $L^2(\mathcal{G}^\perp_r,P)$, where the $\sigma$-field $\mathcal{G}^\perp_r$ is generated by the random variables in $\vro$, (and is, thus, orthogonal to $\E\left[\xi | \mathcal{F}_r^X\right]$), we get
\begin{eqnarray*}
\left(S\left(\xi -  \E\left[\xi | \mathcal{F}_r^X\right]\right)\right)(\hh) &=&\left(S\xi\right)(\hh) \\
&=& \left(S\xi\right)\left(\cc^r - \cc\right) = \E \left[\xi e^{\diamond I\left(\cc^r - \cc\right)}\right]\\
&=& \left(S\E \left[\xi e^{\diamond I\left(\cc^r - \cc\right)}\right]\right)(\hh).
\end{eqnarray*}
The Wick exponentials $\{e^{\diamond I(\hh)}, I(\hh) \in \vro\}$ form a total subset of 
$L^2(\mathcal{G}^\perp_r,P)$, and hence the $S$-transform determines
random variables in this space uniquely. In particular,
$$
\xi -  \E\left[\xi | \mathcal{F}_r^X\right] = \E \left[\xi e^{\diamond I\left(\cc^r - \cc\right)}\right],
$$
or equivalently
$$
\xi = \E\left[\xi | \mathcal{F}_r^X\right] + \E \left[\xi e^{\diamond I\left(\cc^r - \cc\right)}\right].
$$
Thus $\xi$ is $\mathcal{F}_r^X$-measurable.
\end{proof}
\begin{cor}
Let $\xi \in D^{\cc}_r$. Then $\widetilde{E}_r^{\cc}\left[\xi | \mathcal{F}_r^X\right]$ is 
$\mathcal{F}_r^X$-measurable.
\end{cor}
\begin{proof}
Recalling that $\I^*_r\circ\I^*_r=\I^*_r$, we get, for every $\hh\in \HH$,
$$
S\left( \widetilde{E}_r^{\cc}\left[\xi | \mathcal{F}_r^X\right]\right)\left((\hh+\cc)^{r} - \cc\right)=
S(\xi)\left(\left((\hh+\cc)^{r}\right)^r - \cc\right)=S(\xi)\left((\hh+\cc)^{r} - \cc\right).
$$
Hence, $\widetilde{E}_r^{\cc}\left[\xi | \mathcal{F}_r^X\right]\in D_r^{\cc}$ and
$$
\widetilde{E}_r^{\cc}\left[\xi | \mathcal{F}_r^X\right]=\widetilde{E}_r^{\cc}\left[\widetilde{E}_r^{\cc}\left[\xi | \mathcal{F}_r^X\right]| \mathcal{F}_r^X\right].
$$
The assertion now is a direct consequence of Theorem \ref{lma:measurable}.
\end{proof}

We now explain how to compute shifted quasi-conditional expectations of some generalized Skorokhod integrals.
\begin{pro}
\label{lma:removing_Z_part}
Suppose $X$ has an indefinite Wiener integral, $\mathcal{H}$ is  appropriate for $X$ satisfying (I$_{\mathcal{H}}$),
and the generalized Skorokhod
integral $\int_t^a Z \ud^{\diamond_{\mathcal{H}}}X$ exists for some $Z\in L^2(\Omega,\mathcal{H})$ and some  $0\leq t\leq a\leq T$.
Moreover, assume that  $\cinth \in CM$ is such that $\mathcal{H}$ is a space of integrands for $\cinth$. Then, for every $v\leq t$,
$$
\widetilde{E}^{\cc}\left[\left.\int_t^a Z\ud^{\diamond_{\mathcal{H}}}X + \int_t^a Z \ud^{\mathcal{H}} \cinth \right| \mathcal{F}^X_v\right] = 0.
$$
\end{pro}
\begin{proof}
Denote 
$$
\xi = \int_t^a Z\ud^{\diamond_{\mathcal{H}}}X + \int_t^a Z \ud^{\mathcal{H}} \cinth.
$$
As $\Cset_{\mathcal{H}}$ is dense in $\HH_X$, $\I^*_v$ is continuous, and the map $\hh\mapsto e^{\diamond I(\hh)}$ is continuous, it suffices to show
$$
S(\xi)((\hh+\cc)^v - \cc )=0
$$
for every $\hh\in \Cset_{\mathcal{H}}$. For $\hh\in \Cset_{\mathcal{H}}$, write $h^v=I(\hh^v)$. Then, it is easy to check, that
$\hh \in  \Cset_{\mathcal{H}}$ implies $\hh^v \in \Cset_{\mathcal{H}}$ and, for ${\bf g}\in \mathcal{H}$,
$$
\int_0^T {\bf g} \,\ud^{\mathcal{H}} {\inth}^v= \int_0^T \I^{\mathcal{H}}_v({\bf g}) \,\ud^{\mathcal{H}} {\inth}.
$$
Indeed, this follows from
\begin{eqnarray*}
 \int_0^T {\bf 1}_{(0,t]} \,\ud^{\mathcal{H}} {\inth}^v&=& {\inth}^v(t)=\langle {\bf 1}_{(0,t]}, \hh^v\rangle_{\HH}= \langle {\bf 1}_{(0,t\wedge v]}, \hh\rangle_{\HH}
  = {\inth}(t\wedge v)\\ &=& \int_0^T  \I^{\mathcal{H}}_v\left({\bf 1}_{(0,t]}\right)\,\ud^{\mathcal{H}} {\inth},\quad 0\leq t \leq T,
\end{eqnarray*}
in conjunction with the continuity of $\I^{\mathcal{H}}_v$ and $\int_0^T   \cdot \,\ud^{\mathcal{H}} {\inth}$. 
In particular, we observe that $\hh^{\cc}_v := (\hh+\cc)^v - \cc \in \Cset_{\mathcal{H}}$ for $\hh\in \Cset_{\mathcal{H}}$, 
because $\cc \in\Cset_{\mathcal{H}}$ by assumption.

Now, let us first assume that $Z$ is a simple integrand of the form 
$Z=\sum_{k=1}^ne^{\diamond \fc_k} \textbf{1}_{(a_k,T]} \in L^2(\Omega,\mathcal{H})$ for some numbers $a_k\in[0,T]$ and $\fc_k \in H_X$. We have
$$
\I^{\mathcal{H}}_{(t,a]} (Z) = \sum_{k=1}^n e^{\diamond \fc_k} \1_{(a_k \vee t,a]}
$$
and consequently, taking $S$-transform of $\xi$ at $\hh^{\cc}_v$ yields, 
for $\hh\in \Cset_{\mathcal{H}}$, by \eqref{eq:S_transform_first_chaos2} and \eqref{eq:simple_integral_S_transform}, 
\begin{eqnarray*}
\left(S\xi\right)\left(\hh^{\cc}_v\right)&=&\sum_{k=1}^n \left(Se^{\diamond \fc_k}\right)\left(\hh^{\cc}_v\right)\left(S\left(X_a-X_{a_k\vee t}\right)\right)\left(\hh^{\cc}_v\right)\\
&&+\sum_{k=1}^n \left(Se^{\diamond \fc_k}\right)\left(\hh^{\cc}_v\right)\left(S\left(X_a-X_{a_k\vee t}\right)\right)\left(\cc\right)\\
&=&\sum_{k=1}^n \left(Se^{\diamond \fc_k}\right)\left(\hh^{\cc}_v\right)\left(S\left(X_a-X_{a_k\vee t}\right)\right)\left(\left(\hh+\cc\right)^v\right) = 0
\end{eqnarray*}
thanks to the fact that 
$$
\left(S\left(X_a-X_{a_k\vee t}\right)\right)\left(\left(\hh+\cc\right)^v\right) = \la \textbf{1}_{(a_k\vee t,a]},(\hh+\cc)^v\ra_{\HH} 
= \la \I^{\mathcal{\HH}}_v\1_{(a_k\vee t,a]},\hh+\cc\ra_{\HH} = 0
$$ 
since $v\leq t$. The general case $Z \in L^2(\Omega,\mathcal{H})$ now follows directly by approximating with simple integrands
of form $Z_n = \sum_{k=1}^ne^{\diamond \fc_k^{(n)}} \textbf{1}_{(a_k^{(n)},T]}$ converging to $Z$ in $L^2(\Omega,\mathcal{H})$ together with equation \eqref{eq:S_transform_general_integral_convergence}. 
\end{proof}

\begin{rem}
Suppose, under the assumptions of Proposition \ref{lma:removing_Z_part},
that
$\int_0^t Z \ud^{\mathcal{H}}X$ exists for every $t\in[0,T]$. We also assume that $Z$ is \emph{quasi-adapted} in the sense that there is a sequence 
of $(\mathcal{F}^X_t)_{t\in [0,T]}$-adapted processes $(Z^n)$ in $\mathcal{E}$,
which converges to $Z$ in $L^2(\Omega,\mathcal{H})$.
 Then, thanks to the quasi-adaptedness of $Z$, 
$$
\widetilde{E}^{\cc}\left[\left.\int_0^s Z\ud^{\diamond_{\mathcal{H}}}X + \int_0^s Z \ud^{\mathcal{H}} \cinth \right| \mathcal{F}^X_s\right] = \int_0^s Z\ud^{\diamond_{\mathcal{H}}}X + \int_0^s Z \ud^{\mathcal{H}} \cinth,\quad 0\leq s\leq T.
$$
Indeed, this equation holds with $Z$ replaced by $Z^n$ by the definiton of the Skorokhod integral for simple integrands in (\ref{eq:simple_integral}) and Theorem \ref{lma:measurable}. It then carries over 
to $Z$ by the limiting argument in (\ref{eq:S_transform_general_integral_convergence}). 
Hence, in view of Proposition \ref{lma:removing_Z_part}, we obtain, for $0\leq s\leq t\leq T$, by linearity of the shifted 
quasi-conditional expectation,
$$
\widetilde{E}^{\cc}\left[\left.\int_0^t Z\ud^{\diamond_{\mathcal{H}}}X + \int_0^t Z \ud^{\mathcal{H}} \cinth \right| \mathcal{F}^X_s\right] = \int_0^s Z\ud^{\diamond_{\mathcal{H}}}X + \int_0^s Z \ud^{\mathcal{H}} \cinth.
$$
This property is analogous to the martingale property of It\^o integrals with drift under a change of measure induced by a Cameron-Martin shift
in the Brownian motion case. It has already been observed for fractional Brownian motion with $H>1/2$ in the non-shifted case by \cite{hu-oksendal-sulem}. Note that in our context we cannot apply the classical notion of adaptedness, 
because $Z$ need not be a stochastic process in time.
\end{rem}

We have shown that quasi-conditional expectation operator 
shares many properties with the classical conditional expectation operator. 
The next result reveals that, as an important difference, Jensen's inequality does not hold for nonmartingales. This result 
turns out to be one important building block of our non-existence theorem, cp. the construction in Example \ref{exa:not_in_domain}
below.

\begin{thm}
\label{thm:norm_increase}
Suppose $X$ has an indefinite Wiener integral. Then the following are equivalent;
\begin{enumerate}
\item $X$ is not a martingale,
\item There are $\fc \in H_X$ and  $r\in[0,T]$ such that $\E \left[\widetilde{E}\left[\fc | \mathcal{F}_r^X\right]^2\right] > \E \left[\fc^2\right]$.
\end{enumerate}
\end{thm}
\begin{proof}
\textbf{(ii)} $\Rightarrow$ \textbf{(i)}: We prove the contraposition. Hence, suppose $X$ is a martingale. 
Then, thanks to Theorem \ref{thm:martingale_change_of_measure}, we have $\widetilde{E}\left[\fc | \mathcal{F}_r^X\right] = \E\left[\fc | \mathcal{F}_r^X\right]$, and 
classical conditional expectation satisfies, of course, Jensen's inequality.
\\ 
\textbf{(i)} $\Rightarrow$ \textbf{(ii)}: Denote by $d_r$ the number
$$
d_r  := \sup_{\Psi\in \vr, \Upsilon\in \vT, \Vert \Psi\Vert = \Vert \Upsilon\Vert=1} \E\left[\Psi\Upsilon\right]
$$
and note that, since $X$ is not a martingale, we have $d_r>0$ for some $r\in [0,T]$. Indeed, if $d_r = 0$ for all $r\in[0,T]$, then 
the increments of $X$ are uncorrelated and, thus, independent by Gaussianity of $X$. 
Hence $X$ is a martingale which would lead to a contradiction. For fixed $\epsilon>0$ small,
choose some $\Upsilon \in \vr$ and $\Psi \in \vT$ such that 
$\E \left[\Upsilon^2\right] = \E\left[\Psi^2\right] = 1$ and $\E\left[\Upsilon\Psi\right] \geq d_r -\epsilon$. 
Define also $\fc = \Upsilon - d_r \Psi$ and note that, thanks to Theorem \ref{thm:quasi_equivalent} (ii), $\fc\neq 0$, i.e.,
$
\E\left[\fc^2\right] = 1 + d_r^2 - 2d_r \E\left[\Upsilon\Psi\right]>0.
$
By (\ref{eq:quasi_first_chaos})  
we, moreover, have
$
\widetilde{E}\left[\fc|\mathcal{F}_r^X\right] = \Upsilon.
$
Hence,
\begin{equation*}
\begin{split}
\E\left[\widetilde{E}\left[\fc|\mathcal{F}_r^X\right]^2\right]&= 1  \\
&= \frac{1}{1+d_r^2 - 2d_r\E\left[\Upsilon\Psi\right]}\E[\fc^2] \\
&\geq \frac{1}{1-d_r^2 + 2d_r\epsilon}\E[\fc^2].
\end{split}
\end{equation*}
Recalling that $d_r>0$, we get $(1-d_r^2 + 2d_r\epsilon)^{-1}\rightarrow (1-d_r^2)^{-1}>1$ as $\epsilon$ goes to zero. Hence, for sufficiently 
small $\epsilon>0$, 
$$
\E\left[\widetilde{E}\left[\fc|\mathcal{F}_r^X\right]^2\right]>\E[\fc^2].
$$
\end{proof}

We now proceed to give a characterisation of the domain $D^{\cc}_r$ 
in terms of the chaos decomposition, which we recall first.
 For $q\geq 1$, the $q$th Wiener chaos of $X$ is 
defined as the closed linear subspace of $L^2_X$
generated by the family $\{H_{q}(I(\hh)) : \hh\in  \HH,\left\| \hh\right\| _{ \HH}=1\}$, 
where $H_{q}$ is the $q$th Hermite polynomial. The mapping $I_{q}(\hh^{\otimes q})=H_{q}(I(\hh))$ can be extended to a
linear isometry between the symmetric tensor product $ \HH^{\tilde\otimes q}$
and the $q$th Wiener chaos, and for any $\hh \in  \HH^{\tilde\otimes q}$ the random variable $I_q(\hh)$ is called a multiple Wiener integral of order $q$. It is known that each element $\xi \in L^2_X$ has a unique chaos decomposition
\begin{equation}
\label{eq:chaos_decomposition}
\xi = \sum_{k=0}^\infty I_k(\ff_k), \quad \ff_k \in \HH^{\tilde\otimes k},
\end{equation}
where $I_0$ denotes the identity on $H_0=\mathbb{R}$ and $\ff_0=\E[\xi]$. For more details, we refer to Janson \cite{janson} and Nualart \cite{nu}.
For handy reference we just note the chaos decomposition of a Wick exponential
\begin{equation}\label{eq:chaos_exponential}
e^{\diamond I(\hh)} = 1 +\sum_{k=1}^\infty I_k\left(\frac{1}{k!}\hh^{\otimes k}\right),\quad \hh\in\HH_X,
\end{equation}
and the isometry property
\begin{equation}\label{eq:chaos_isometry}
 \E\left[\left(\sum_{k=0}^\infty I_k(\ff_k) \right)\left(\sum_{k=0}^\infty I_k(\hh_k)\right)\right]=\sum_{k=0}^\infty
 k! \langle \ff_k, \hh_k \rangle_{\HH_X^{\otimes k}}.
\end{equation}

 We begin with the following technical lemma.
\begin{lma}
\label{lma:convergence}
Suppose $X$ has an indefinite Wiener integral, 
let $\ff_k \in \HH^{\tilde \otimes k}$ be a sequence such that $\sum_{k=0}^\infty k!\Vert \ff_k\Vert_{\HH^{\otimes k}}^2 < \infty$, and $\cc \in \HH$ be given. Denote 
$\cc_r = \Tr^*_r \cc-\cc \in \HH$, and for given $i\leq k$ set 
\begin{equation}
\label{eq:f_map_smaller}
\CC_{r,k,i} \ff_k = \la \ff_k, \cc_r^{\otimes k-i}\ra_{\HH^{\otimes k-i}}.
\end{equation}
Then the series
\begin{equation}
\label{eq:new_coefficient}
\widetilde{\ff}_n^{r,\cc} = \sum_{k=n}^\infty {k \choose n} \Tr_r^{\otimes n}\CC_{r,k,n} \ff_k
\end{equation}
converges in $\HH^{\tilde \otimes n}$. 
\end{lma}
\begin{proof}
Standard estimates for tensor powers of operators (see, e.g., Proposition E.20 in \cite{janson}) yield
\begin{equation}
\label{eq:T_bound}
\Vert \Tr_r^{\otimes n}\CC_{r,k,n} \ff_k\Vert_{\HH^{\otimes n}} \leq \Vert \Tr_r\Vert_{op}^n \Vert \cc_r\Vert_{\HH}^{k-n} \Vert \ff_k\Vert_{\HH^{\otimes k}},
\end{equation}
where we recall that $\Vert \cdot \Vert_{op}$ stands for the operator norm.
Let now $n$ be fixed. For every $m>l+n>n$ we get
\begin{equation*}
\begin{split}
\frac{1}{\sqrt n}\Vert\sum_{k=n+l}^m {k \choose n} \Tr_r^{\otimes n}\CC_{r,k,n} \ff_k\Vert_{\HH^{\tilde\otimes n}} 
& =\Vert\sum_{k=n+l}^m {k \choose n} \Tr_r^{\otimes n}\CC_{r,k,n} \ff_k\Vert_{\HH^{\otimes n}} \\ & \leq \Vert \Tr_r\Vert_{op}^n 
\sum_{k=n+l}^m{k \choose n}\Vert \cc_r\Vert_{\HH}^{k-n} \Vert \ff_k\Vert_{\HH^{\otimes k}}\\
&= \frac{\Vert \Tr_r\Vert_{op}^n}{n!} \sum_{k=n+l}^m \sqrt{k!}\Vert \ff_k\Vert_{\HH^{\otimes k}} \frac{\sqrt{k!}}{(k-n)!}\Vert \cc_r\Vert_{\HH}^{k-n}\\
&\leq \frac{\Vert \Tr_r\Vert_{op}^n}{n!} \sqrt{\sum_{k=n+l}^m k!\Vert \ff_k\Vert^2_{\HH^{\otimes k}}}\sqrt{\sum_{k=n+l}^m \frac{k!}{(k-n)!^2}\Vert \cc_r\Vert_{\HH}^{2(k-n)}}
\end{split}
\end{equation*}
which shows that $\sum_{k=n}^m {k \choose n} \Tr_r^{\otimes n}\CC_{r,k,n} \ff_k$ is a Cauchy-sequence in $\HH^{\tilde\otimes n}$, since $\sum_{k=n}^\infty k!\Vert \ff_k\Vert^2_{\HH^{\otimes k}}$ converges 
by assumption and 
$$
\sum_{k=n}^\infty \frac{k!}{(k-n)!^2}\Vert \cc_r\Vert_{\HH}^{2(k-n)} = \sum_{k=0}^\infty \frac{(k+n)!}{k!^2}\Vert \cc_r\Vert_{\HH}^{2k} < \infty.
$$
\end{proof}
\begin{thm}[Characterisation of $D_r^{\cc}$]\label{thm:domain}
\label{cor:characterisation}
Suppose $X$ has an indefinite Wiener integral. Then, for every $r\in[0,T]$, the domain $D_r^{\cc}$ of the operator $\widetilde{E}^{\cc}\left[\cdot |\mathcal{F}_r^X\right]$ is given by
\begin{equation}
\label{eq:Ec_domain}
D_r^{\cc} =\left\{\xi = \sum_{k=0}^\infty I_k(\ff_k) : \sum_{k=0}^\infty k! \Vert \widetilde{\ff}^{r,\cc}_{k}\Vert^2_{\HH^{\otimes k}} < \infty\right\},
\end{equation}
where $\widetilde{\ff}^{r,\cc}_k$ are given by \eqref{eq:new_coefficient}. 
Furthermore, for $\xi \in D_r^{\cc}$ with chaos decomposition $\xi = \sum_{k=0}^\infty I_k(\ff_k)$ we have
\begin{equation}
\label{eq:quasi_chaos_c}
\widetilde{E}^{\cc}\left[\xi |\mathcal{F}_r^X\right] = \sum_{k=0}^\infty I_k\left(\widetilde{\ff}^{r,\cc}_k\right).
\end{equation}
\end{thm}
\begin{proof}
Denote again 
$\cc_r = \Tr_r^* \cc-\cc \in \HH$, and 
$(\hh+\cc)^r = \Tr_r^*(\hh+\cc) = \Tr_r^*\hh + \Tr_r^*\cc$. 
Then, for a random variable of form 
$I_q(\ff)$, $\ff \in \HH^{\tilde \otimes q}$, we have, by (\ref{eq:chaos_exponential}) and (\ref{eq:chaos_isometry}), 
\begin{eqnarray*}
&& \left(SI_q(\ff)\right)((\hh+\cc)^r-\cc) = \la \ff, \left(\Tr_r^*\hh+\cc_r\right)^{\otimes q}\ra_{\HH^{\otimes q}}\\
&=& \sum_{k=0}^q {q\choose k}\la \ff,\left(\Tr_r^* \hh\right)^{\otimes k} \otimes \left(\cc_r\right)^{\otimes q-k}\ra_{\HH^{\otimes q}}
=  \sum_{k=0}^q {q\choose k} \la \CC_{r,q,k}\ff,\left(\Tr_r^* \hh\right)^{\otimes k}\ra_{\HH^{\otimes k}}\\
&=&  \sum_{k=0}^q {q\choose k} \la \Tr_r^{\otimes k}\CC_{r,q,k}\ff,\hh^{\otimes k}\ra_{\HH^{\otimes k}}
= \sum_{k=0}^q {q\choose k}  \E\left[I_k(\Tr_r^{\otimes k}\CC_{r,q,k}\ff)e^{\diamond I(\hh)}\right]\\
&=& S\left(\sum_{k=0}^q {q\choose k}I_k(\Tr_r^{\otimes k}\CC_{r,q,k}\ff)\right)(\hh).
\end{eqnarray*}
Hence, $I_q(\ff)\in D^\cc_r$ and 
\begin{equation}\label{eq:qc_fixedchaos}
\widetilde{E}^{\cc}\left[I_q(\ff)|\mathcal{F}_r^X\right] = \sum_{k=0}^q {q\choose k}I_k(\Tr_r^{\otimes k}\CC_{r,q,k}\ff).
\end{equation}
Denote now
$
\xi_n = \sum_{k=0}^n I_k(\ff_k).
$
Then $\xi_n$ converges to $\xi$ in $L^2$. Furthermore, $\xi_n \in D^\cc_r$ 
and
$$
\widetilde{E}^{\cc}\left[\xi_n |\mathcal{F}_r^X\right] = \sum_{j=0}^n I_j\left(\sum_{k=j}^n {k \choose j} \Tr_r^{\otimes j}\CC_{r,k,j}\ff_k\right)
$$
by (\ref{eq:qc_fixedchaos}) and linearity. 
Consequently, if 
\begin{equation}
\label{eq:quasi_chaos_cond}
\sum_{k=0}^\infty k! \Vert \widetilde{\ff}^{r,\cc}_{k}\Vert^2_{\HH^{\otimes k}} < \infty
\end{equation}
where $\widetilde{\ff}^{r,\cc}_k$ is given by \eqref{eq:new_coefficient}, 
then $\widetilde{E}^{\cc}\left[\xi_n |\mathcal{F}_r^X\right]$ converges in $L^2_X$ to 
$
 \sum_{k=0}^\infty I_k\left(\widetilde{\ff}^{r,\cc}_k\right)
$
from which $\xi \in D^\cc_r$ and \eqref{eq:quasi_chaos_c} 
follows by closedness of $\widetilde{E}^{\cc}\left[\cdot |\mathcal{F}_r^X\right]$.

Assume now that $\xi \in D_r^{\cc}$ with chaos decomposition 
$
\xi = \sum_{k=0}^\infty I_k(\ff_k),
$
and let $\hh \in \HH$. Applying the computations above we get
$$
\left(S \xi \right)((\hh+\cc)^r-\cc) =\sum_{k=0}^\infty\la \ff_k, \left(\Tr^*_r\hh+\cc_r\right)^{\otimes k}\ra_{\HH^{\otimes k}}
= \sum_{k=0}^\infty \sum_{j=0}^k {k \choose j} \la \Tr_r^{\otimes j}\CC_{r,k,j}\ff_{k}, \hh^{\otimes j}\ra_{\HH^{\otimes j}}.
$$
By \eqref{eq:T_bound} we obtain 
\begin{equation*}
\begin{split}
& \sum_{k=0}^\infty \sum_{j=0}^k \left|{k \choose j} \la \Tr_r^{\otimes j}\CC_{r,k,j}\ff_{k}, \hh^{\otimes j}\ra_{\HH^{\otimes j}}\right|\\
&\leq \sum_{k=0}^\infty \Vert \ff_k\Vert_{\HH^{\otimes k}} \sum_{j=0}^k {k\choose j} \Vert \hh\Vert_{\HH}^j \Vert \Tr_r\Vert_{op}^j \Vert \cc_r\Vert_{\HH}^{k-j}\\
&= \sum_{k=0}^\infty \Vert \ff_k\Vert_{\HH^{\otimes k}} \left(\Vert \hh\Vert_{\HH} \Vert \Tr_r\Vert_{op} + \Vert \cc_r\Vert_{\HH} \right)^k \\
&\leq \sqrt{\sum_{k=0}^\infty k! \Vert \ff_k\Vert_{\HH^{\otimes k}}^2} \sqrt{\sum_{k=0}^\infty \frac{1}{k!}\left(\Vert \hh\Vert_{\HH} \Vert \Tr_r\Vert_{op} + \Vert \cc_r\Vert_{\HH} 
\right)^{2k}}\\
& < \infty.
\end{split}
\end{equation*}
Hence we can change the order of summation, and consequently we get
\begin{equation*}
\left(S \xi \right)((\hh+\cc)^r-\cc) = \sum_{j=0}^\infty \sum_{k=j}^\infty {k \choose j} \la \Tr_r^{\otimes j}\CC_{r,k,j}\ff_{k}, \hh^{\otimes j}\ra_{\HH^{\otimes j}} = \sum_{j=0}^\infty \la \widetilde{\ff}_j^{r,\cc},\hh^{\otimes j}\ra_{\HH^{\otimes j}}
\end{equation*}
by Lemma \ref{lma:convergence}. 
On the other hand, let the chaos decomposition of $\widetilde{E}^{\cc}\left[\xi |\mathcal{F}_r^X\right]$ be given by
$$
\widetilde{E}^{\cc}\left[\xi |\mathcal{F}_r^X\right] = \sum_{k=0}^\infty I_k\left(\widehat{\ff}_k\right).
$$
We have
\begin{equation*}
\left(S\widetilde{E}\left[\xi |\mathcal{F}_r^X\right]\right)(\hh) = \sum_{k=0}^\infty \la \widehat{\ff}_k,\hh^{\otimes k}\ra_{\HH^{\otimes k}}.
\end{equation*}
In particular, for $ \hh \in \HH$ and $\alpha\in \mathbb{R}$, we get
\begin{equation}
\label{eq:chaos_dec_direct}
\left(S \xi \right)((\alpha\hh+\cc)^r-\cc) = \sum_{k=0}^\infty \alpha^k \la \widetilde{\ff}_k^{r,\cc},\hh^{\otimes k}\ra_{\HH^{\otimes k}}
\end{equation}
and
\begin{equation}
\label{eq:chaos_dec_quasi}
\left(S\widetilde{E}\left[\xi |\mathcal{F}_r^X\right]\right)(\alpha\hh) = \sum_{k=0}^\infty \alpha^k\la \widehat{\ff}_k,\hh^{\otimes k}\ra_{\HH^{\otimes k}}
\end{equation}
Comparing power series \eqref{eq:chaos_dec_direct} and \eqref{eq:chaos_dec_quasi} we obtain
$\la \widetilde{\ff}_k^{r,\cc},\hh^{\otimes k}\ra_{\HH^{\otimes k}} = \la \widehat{\ff}_k,\hh^{\otimes k}\ra_{\HH^{\otimes k}}$ 
for every $k\geq 0$ and every $\hh \in \HH$. This implies $\widehat{\ff}_k = \widetilde{\ff}^{r,\cc}_k, k\geq 0$ since 
$\left\{\hh^{\otimes k}: \hh \in \HH\right\}$ is total in $\HH^{\tilde\otimes k}$. Consequently, by (\ref{eq:chaos_isometry}),
$$
\sum_{k=0}^\infty k! \Vert \widetilde{\ff}^{r,\cc}_{k}\Vert^2_{\HH^{\otimes k}} = \sum_{k=0}^\infty k! \Vert \widehat{\ff}_{k}\Vert^2_{\HH^{\otimes k}}=\E\left[\widetilde{E}^{\cc}\left[\xi |\mathcal{F}_r^X\right]^2 \right]<\infty.  
$$
\end{proof}
\begin{rem}
\label{rem:non-shifted}
For the non-shifted operator $\widetilde{E}\left[\cdot|\mathcal{F}_r^X\right]$ we have $\widetilde{\ff}_n^{r,0} = \Tr_r^{\otimes n}\ff_n$ so that
\begin{equation}
\label{eq:E_domain}
D_r =\left\{\xi = \sum_{k=0}^\infty I_k(\ff_k) : \sum_{k=0}^\infty k!\Vert \Tr_r^{\otimes k}\ff_k\Vert^2_{\HH^{\otimes k}} < \infty\right\},
\end{equation}
and for $\xi \in D_r$ with chaos decomposition $\xi = \sum_{k=0}^\infty I_k(\ff_k)$ we have
\begin{equation}
\label{eq:quasi_chaos}
\widetilde{E}\left[\xi |\mathcal{F}_r^X\right] = \sum_{k=0}^\infty I_k\left(\Tr^{\otimes k}_r\ff_k\right).
\end{equation}
This shows, that even in the special case, where $X$ is a fractional Brownian motion with Hurst parameter $H>1/2$, our definition of
quasi-conditional expectation
extends the one by \cite{hu-oksendal}, cp. also \cite{bender-elliott} and \cite{hu-peng}.
\end{rem}

We can now continue with a construction of a random variable which does not belong to the domain of shifted quasi-conditional 
expectation in the non-martingale case.

\begin{exa}
\label{exa:not_in_domain}
Let $\cc \in \HH$ be fixed and suppose $X$ is not a martingale. Then by Theorem \ref{thm:norm_increase}
and Remark \ref{rem:quasi_indefinite}, $\Vert \I_r\Vert_{op} >1$ for some $r\in [0,T]$, so that there 
exists $\ff\in \HH$ and $r\in[0,T]$ such that $\Vert \Tr_r \ff\Vert_{\HH} > 1 > \Vert \ff\Vert_{\HH}$. Without loss of generality
(by changing to $-\ff$ if necessary) we can also assume $\la \ff,\cc_r\ra_{\HH} \geq 0$. 
Defining $\xi = \sum_{k=0}^\infty I_k\left(\frac{1}{\sqrt{k!}} \ff^{\otimes k}\right)$ we have, in view of (\ref{eq:chaos_isometry}), $\xi \in L^2_X$, while 
using the fact $\la \ff,\cc_r\ra_{\HH} \geq 0$ we get 
\begin{equation*}
\begin{split}
\sum_{k=0}^\infty k! \Vert \widetilde{\ff}^{r,\cc}_k\Vert^2_{\HH^{\otimes k}} &= 
\sum_{k=0}^\infty k! \Vert \Tr_r\ff \Vert_{\HH}^{2k}\left(\sum_{j=k}^\infty {j \choose k} \frac{1}{\sqrt{j!}} \la \ff,\cc_r\ra_{\HH}^{j-k} \right)^2\\
&\geq \sum_{k=0}^\infty \Vert \Tr_r\ff \Vert_{\HH}^{2k} \\
&= \infty.
\end{split}
\end{equation*}
Hence $\xi \notin D_r^{\cc}$.
\end{exa}

\begin{rem}
 The construction in the above example is analogous to that in Theorem 5.2 in \cite{bender-elliott} for the case of fractional Brownian
 motion with Hurst parameter $H>1/2$ and $\cc=0$. Instead of building on an abstract existence argument for $\ff$ with  
 $\Vert \Tr_r \ff\Vert_{\HH} > 1 > \Vert \ff\Vert_{\HH}$, the authors in \cite{bender-elliott} give a simple explicit example of such $\ff$ in the fractional Brownian motion case.
\end{rem}

In the following theorem we have summarised our main results concerning quasi-conditional 
expectation and its domain.
\begin{thm}
\label{thm:mg_equivalence}
Suppose $X$ has an indefinite Wiener integral. Then the following are equivalent:
\begin{enumerate}
\item $X$ is a martingale,
\item for all $r\in[0,T]$ we have $\Vert \I_r\Vert_{op} = 1$, where $\Vert \cdot \Vert_{op}$ denotes the operator norm,
\item for all $r\in[0,T]$, $\widetilde{E}\left[\cdot | \mathcal{F}_r\right]$ is the restriction of the classical conditional
expectation $\E\left[\cdot | \mathcal{F}_r\right]$ to $L^2_X$,
\item for all $r\in[0,T]$ we have $D_r = L^2_X$.
\end{enumerate}
\end{thm}
\begin{proof}
Using Remark \ref{rem:quasi_indefinite} the equivalence \textbf{(i)} $\Leftrightarrow$ \textbf{(ii)} is just a
reformulation of Theorem \ref{thm:norm_increase} in terms of operator norm $\Vert \I_r\Vert_{op}$. Furthermore,
implication \textbf{(i)} $\Rightarrow$ \textbf{(iii)} follows from Theorem
\ref{thm:martingale_change_of_measure} and implication \textbf{(iii)} $\Rightarrow$ \textbf{(iv)} is obvious.
Finally, Example \ref{exa:not_in_domain} provides the implication 
\textbf{(iv)} $\Rightarrow$ \textbf{(i)}. 
\end{proof}

We end this section by proving our main results Theorem \ref{thm:mg_domain_connection} and Theorem   \ref{thm:solution_Y}.
\begin{proof}[Proof of Theorem \ref{thm:mg_domain_connection}]
\textbf{(ii)} $\Rightarrow$ \textbf{(i)}: Suppose $X$ is a martingale. Then the choice $c=0$ leads to contradiction with the help of Theorem \ref{thm:mg_equivalence}.

\textbf{(i)} $\Rightarrow$ \textbf{(ii)}: Let $\cc \in \HH$ be fixed and suppose $X$ is not a martingale. Then Example \ref{exa:not_in_domain} provides an element $\xi \in L^2_X$ such that $\xi \notin D_r^{\cc}$.
\end{proof}

\begin{proof}[Proof of Theorem \ref{thm:solution_Y}]
Suppose $(Y,Z,\mathcal{H})$ is a mild solution to (\ref{BSDE_differentialform}).   By the definition of shifted quasi-conditional expectation we have to show that
\begin{equation}\label{eq:hilf0001}
\left(S \tilde{\xi}\right)(\left(\hh+\cc\right)^t - \cc) = A(t)\left(SY_t\right)(\hh) - \int_0^t A(u)\left(SG_u\right)(\hh)\ud\gamma(u)
\end{equation}
for every $t\in [0,T]$ and $\hh \in \HH_X$. We here recall that
$$
\tilde{\xi} = \xi - \int_0^T A(s)G_s\ud\gamma(s),\quad A(t) = e^{\int_t^T a(s) \ud \gamma(s)}.
$$
Let $v\in[0,T]$ and $\hh\in \HH_X$ be fixed. Taking $S$-transform at $\hh^{\cc}_v=(\hh+\cc)^v - \cc $ in \eqref{eq:BSDE_abstract},
using linearity together with Proposition \ref{lma:removing_Z_part} 
yields, for $t\in[v,T]$,
\begin{equation}
\label{eq:S_Y}
\left(SY_t\right)\left(\hh^{\cc}_v\right) = 
 \left(S \xi\right)\left(\hh^{\cc}_v\right) 
- \int_t^T \left[a(u)\left(SY_u\right)\left(\hh^{\cc}_v\right) + \left(SG_u\right)
\left(\hh^{\cc}_v\right)\right]\ud \gamma(u).
\end{equation}
Since $\gamma$ is a continuous function of bounded variation, so is $A(t)$. Moreover, for fixed $v$ the function $t\mapsto \left(SY_t\right)\left(\hh^{\cc}_v\right)$ is continuous and of bounded variation on $[v,T]$
by \eqref{eq:S_Y}. Hence
we can use integration by parts to get, for $t\in[v,T]$,
\begin{equation*}
A(t)\left(SY_t\right)\left(\hh^{\cc}_v\right)= A(T)\left(SY_T\right)\left(\hh^{\cc}_v\right) - 
\int_t^T A(u)\ud \left(SY_u\right)\left(\hh^{\cc}_v\right) 
- \int_t^T \left(SY_u\right)\left(\hh^{\cc}_v\right)\ud A(u)
\end{equation*}
from which, using $\ud A(u) = -a(u)A(u)\ud \gamma(u)$ and \eqref{eq:S_Y} again, we obtain 
\begin{equation*}
\begin{split}
 A(t)\left(SY_t\right)\left(\hh^{\cc}_v\right) &=  A(T)\left(SY_T\right)\left(\hh^{\cc}_v\right) - \int_t^T A(u)\left[a(u)\left(SY_u\right)\left(\hh^{\cc}_v\right) + 
\left(SG_u\right)\left(\hh^{\cc}_v\right) \right]\ud \gamma(u)\\
&\quad+ \int_t^T \left(SY_u\right)\left(\hh^{\cc}_v\right)a(u)A(u)\ud \gamma(u)\\
&=  A(T)\left(SY_T\right)\left(\hh^{\cc}_v\right) - \int_t^T A(u)\left(SG_u\right)\left(\hh^{\cc}_v\right)\ud \gamma(u)\\
&= \left(S\tilde\xi\right)\left(\hh^{\cc}_v\right) + \int_0^t A(u)\left(SG_u\right)\left(\hh^{\cc}_v\right)\ud \gamma(u).
\end{split}
\end{equation*}
In particular, taking $t=v$ yields 
\begin{equation*}
A(t)\left(SY_t\right)\left(\hh^{\cc}_t\right) 
= \left(S\tilde\xi\right)\left(\hh^{\cc}_t\right) 
+ \int_0^t A(u)\left(SG_u\right)\left(\hh^{\cc}_t\right)\ud \gamma(u),\quad t\in [0,T].
\end{equation*}
Since $Y$ and $G$ are adapted, we have $\left(SY_t\right)(\left(\hh+\cc\right)^t - \cc) = \left(SY_t\right)(\hh)$ 
and $\left(SG_u\right)(\left(\hh+\cc\right)^t - \cc) = \left(SG_u\right)(\hh)$ for $u\leq t$ by Theorem \ref{lma:measurable},
which finally implies (\ref{eq:hilf0001}). 
\end{proof}

\section{On the existence of an indefinite Wiener integral}\label{sec:indefinite}
\label{subsec:eQC}
In this section, we show, among others, that fractional Brownian motion has an indefinite Wiener integral for the full range of Hurst 
parameters $H\in(0,1)$.

We begin with the following characterisation which shows how the existence of an indefinite Wiener integral can be checked
in terms of the Cameron-Martin space of $X$.
\begin{thm}
\label{thm:quasi_equivalent_to_check}
The following statements are equivalent;
\begin{enumerate}
\item $X$ has an indefinite Wiener integral,
\item The Cameron-Martin space of $X$ is closed against truncation in time, i.e., for every
$\underline{h} \in CM_X$ and $r\in[0,T]$, the mapping $\underline{h}^r :t \mapsto \underline{h}(t\wedge r)$ satisfies $\underline{h}^r\in CM_X$.
\end{enumerate}
\end{thm}
\begin{proof}
\textbf{(i)} $\Rightarrow$ \textbf{(ii)}: Assume that $X$ has an indefinite Wiener integral. Consequently, thanks to the isometry between $H_X$ and $\HH_X$, there exists a bounded 
linear operator $\widetilde{E}_r :H_X \mapsto H_X$ satisfying $\widetilde{E}_r\left(X_t\right) = X_{t\wedge r}$. 
Denote by $\widetilde{E}_r^*$ its adjoint operator. Then,
$$
\inth(t\wedge r)=\E \left[\fc X_{t\wedge r}\right] =\E \left[\fc \widetilde{E}_r[X_t]\right] = \E\left[\widetilde{E}_r^* [\fc] X_t\right]
$$
and consequently, for any $\underline{h} \in CM_X$ the mapping $\underline{h}^r =  \underline{h}(t\wedge r) \in CM_X$, where the associated element for $\underline{h}^r$ is given by $\widetilde{E}_r^*[\fc]$.\\
\textbf{(ii)} $\Rightarrow$ \textbf{(i)}: Suppose that \textbf{(ii)} holds. This means that for each element $\fc \in H_X$ there exists a unique element $\fc^r \in H_X$ such that 
$\E\left[\fc^r X_t\right] = \E\left[\fc X_{t\wedge r}\right]$ for every $t\in [0,T]$. Define a mapping 
$\Phi : H_X \mapsto H_X$ by $\Phi(\fc) = \fc^r$. This operator is linear and closed, and 
hence it is continuous by closed graph theorem. Consequently, the adjoint $\Phi^*$ of $\Phi$ is a bounded linear operator satisfying
$$
\E\left[\fc \Phi^*(X_t)\right] =\E\left[\Phi(\fc) X_t\right]  = \E\left[\fc X_{t\wedge r}\right],
$$
for every $h\in H_X$.
This implies that $\Phi^*(X_t) = X_{t\wedge r}$, and by isometry between $H_X$ and $\HH_X$ there exists a bounded linear operator $\I^{\HH}_r : \HH_X \mapsto \HH_X$ satisfying 
$\I^{\HH}_r\left(\1_{(0,t]}\right) = \1_{(0,t\wedge r]}$. Hence $X$ has 
an indefinite Wiener integral.
\end{proof}

We apply this characterisation to show that fractional Brownian motion $B^H$ has an indefinite Wiener integral. This fact 
and some other related results are the topic of the following theorem.
\begin{thm}
\label{thm:fbm}
\begin{enumerate}
\item Fractional Brownian motion $B^H$ with Hurst parameter $H\in(0,1)$ has an indefinite Wiener integral.
\item If $B^H$ is a fractional Brownian motion with $H>\frac12$, and $\sigma$ is a deterministic function on $[0,T]$ satisfying $\epsilon\leq \sigma\leq \epsilon^{-1}$ for some $\epsilon \in(0,1)$, 
then the fractional Wiener integral $X_t = \int_0^t \sigma(s)\ud B^H_s$ has an indefinite Wiener integral.
\item 
Suppose that $X_t = X^{(1)}_t + \gamma X^{(2)}_t$, where $\gamma \neq 0$ is a constant and $X^{(i)},i=1,2$ are independent centered Gaussian processes. 
If $X^{(1)}$ and $X^{(2)}$ have  indefinite Wiener integrals, then so does $X$.
\end{enumerate}
\end{thm}
\begin{proof}[Proof of Theorem \ref{thm:fbm},(i)]
The case $H>\frac12$ is covered by item \textbf{(ii)} and the case $H=\frac12$ is covered by Theorem \ref{thm:martingale_change_of_measure}. Hence we treat only $H<\frac12$. 

Recall that the left-sided fractional Riemann-Liouville
integral of order $\alpha>0$ is defined as
$$
\left(I_{0+}^\alpha \phi\right)(t) = \frac{1}{\Gamma(\alpha)}\int_0^t \phi(s)(t-s)^{\alpha-1}\ud s, \quad t\in[0,T],
$$ 
for $\phi\in L^1\left([0,T],\ud t\right)$, where $\Gamma(\alpha)$ denotes the Gamma function.
By Remark 3.1 in Decreusefond and \"Ust\"unel \cite{dec-ust},
a function $y(t)$ belongs to the Cameron-Martin space of fractional Brownian motion if and only if
\begin{equation}
\label{eq:fbm_CM}
y(t) = \left(I_{0+}^{H+\frac12} \phi\right)(t)
\end{equation}
for some $\phi \in L^2\left([0,T],\ud t\right)$. By item \textbf{(ii)} of Theorem \ref{thm:quasi_equivalent_to_check},
it suffices to prove the following: For every $\phi \in L^2\left([0,T],\ud t\right)$ and $r\in [0,T]$, there
is a $\phi_r\in L^2\left([0,T],\ud t\right)$ such that for every $t\in[0,T]$
$$
\left(I_{0+}^{H+\frac12} \phi_r\right)(t)=\textbf{1}_{[0,r]}(t) \left(I_{0+}^{H+\frac12} \phi\right)(t)
+ \textbf{1}_{(r,T]}(t) \left(I_{0+}^{H+\frac12} \phi\right)(r)=:y_r(t).
$$
For the rest of the proof we use short notation $\alpha = H+\frac12$. Set 
\begin{equation*}
\phi_r(s) = \begin{cases} 
\phi(s), &\mbox{if } 0\leq s\leq r\\
-\frac{\alpha}{\Gamma(1-\alpha)}\int_0^r \left(I_{0+}^{\alpha}\phi\right)(u)(s-u)^{-1-\alpha}\ud u + C(s-r)^{-\alpha}, &\mbox{if } r<s\leq T
\end{cases}
\end{equation*}
where the constant $C = C(\alpha,r)$ is given by 
$$
C = \frac{1}{\Gamma\left(1-\alpha\right)} \left(I_{0+}^{\alpha} \phi\right)(r).
$$
By Theorem 13.10 and  (2.44) in Samko et al. \cite{s-k-m}
we have $\phi_r\in L^p([0,T],\ud t)$ for $p<\alpha^{-1}=(H+1/2)^{-1}$ and 
$$
\left(I_{0+}^{\alpha}\phi_r\right)(t) = y_r(t).
$$
Hence, it remains to show that $\phi_r \in L^2\left([0,T],\ud t\right)$.  For $x>r$ we have
\begin{equation*}
\begin{split}
\phi_r(x) &= - \frac{\alpha}{\Gamma(1-\alpha)}\int_0^r \left(I^{\alpha}_{0+}\phi\right)(t)(x-t)^{-1-\alpha}\ud t + C(x-r)^{-\alpha} \\
&=  - \frac{\alpha}{\Gamma(1-\alpha)\Gamma(\alpha)}\int_0^r \int_0^t \phi(s)(t-s)^{\alpha-1}\ud s (x-t)^{-1-\alpha}\ud t\\
& + \frac{1}{\Gamma(1-\alpha)\Gamma(\alpha)}\int_0^r \phi(s)(r-s)^{\alpha-1}(x-r)^{-\alpha}\ud s.
\end{split}
\end{equation*}
Following proof of Theorem 13.10 in Samko et al. \cite{s-k-m} we obtain by using Fubini's theorem and change of variable that the first term is given by 
\begin{equation*}
\begin{split}
&- \frac{\alpha}{\Gamma(1-\alpha)\Gamma(\alpha)}\int_0^r \int_0^t \phi(s)(t-s)^{\alpha-1}\ud s (x-t)^{-1-\alpha}\ud t \\
&= -  \frac{1}{\Gamma(1-\alpha)\Gamma(\alpha)}\int_0^r \phi(s)\left(\frac{r-s}{x-r}\right)^{\alpha} \frac{1}{x-s} \ud s.
\end{split}
\end{equation*}
Consequently, for $x>r$,
\begin{equation*}
\begin{split}
\phi_r(x) &= \frac{1}{\Gamma(1-\alpha)\Gamma(\alpha)}\int_0^r \phi(s)\left[(x-r)^{-\alpha}(r-s)^{\alpha-1} - \left(\frac{r-s}{x-r}\right)^{\alpha}\frac{1}{x-s}\right]\ud s\\
&= \frac{1}{\Gamma(1-\alpha)\Gamma(\alpha)}\int_0^r \phi(s)\left(\frac{r-s}{x-r}\right)^{\alpha-1}\frac{1}{x-s}\ud s.
\end{split}
\end{equation*}
Hence $\phi_r(s) \in L^2\left([0,T],\ud t\right)$ by Theorem 1.5 in Samko et al. \cite{s-k-m} (see also the proof of Theorem 13.10
in the same reference) by observing that in this case the kernel $k(x,t)$ of Theorem 1.5 is given by 
$$
k(x,t) = \left(\frac{t}{x}\right)^{\alpha-1}(t+x)^{-1}
$$
which satisfies 
$$
\int_0^\infty |k(x,1)|x^{-\frac12}\ud x = \int_0^\infty |k(1,t)|t^{-\frac12}\ud t < \infty
$$
by the fact $\alpha-1 = H-\frac12\in (-\frac12,0)$. This completes the proof.
\end{proof}

\begin{proof}[Proof of Theorem \ref{thm:fbm},(ii)]
By \cite{bender}, the Cameron-Martin space of $X$ is 
$$
CM_X = \left\{\int_0^{\cdot}\sigma(s)I_{0+}^{H-\frac12}\left(\psi\right)(s)\ud s,\quad \psi \in L^2\left([0,T],\ud t\right)\right\}.
$$
In view of Theorem \ref{thm:quasi_equivalent_to_check}, it suffices to show that, for every $\psi \in L^2\left([0,T],\ud t\right)$ 
and $r\in[0,T]$ 
there is a $\psi_r \in L^2\left([0,T],\ud t\right)$ satisfying
$$
\int_0^{t\wedge r} \sigma(s)I_{0+}^{H-\frac12}\left(\psi\right)(s)\ud s =
\int_0^{t} \sigma(s)I_{0+}^{H-\frac12}\left(\psi_r\right)(s)\ud s 
$$
for every $t\in[0,T]$.
As $H\in(\frac12,1),$ we have $\left(H-\frac12\right)^{-1} >2$, and consequently we
can apply Theorem 13.10 in \cite{s-k-m} to conclude that, given $\psi \in L^2\left([0,T],\ud t\right)$ 
and $r\in[0,T]$, there exists $\psi^r \in L^2([0,T],\ud t)$ such that
$$
\textbf{1}_{[0,r]}(s)I_{0+}^{H-\frac12}\left(\psi\right)(s) = I_{0+}^{H-\frac12}\left(\psi^r\right)(s),\quad 0\leq s\leq T.
$$
 This completes the proof.
\end{proof}
\begin{proof}[Proof of Theorem \ref{thm:fbm},(iii)]
Without loss of generality we can assume that $\gamma=1$. Let 
$$
\fc_n = \sum_{k=1}^n \alpha_k^{(n)} X_{t_k^{(n)}} = \sum_{k=1}^n \alpha_k^{(n)} X^{(1)}_{t_k^{(n)}} + \sum_{k=1}^n \alpha_k^{(n)} X^{(2)}_{t_k^{(n)}} =: \fc_n^{(1)} + \fc_n^{(2)}
$$
be a sequence of simple random variables converging to $\fc \in H_X$ in $L^2_X$.
By taking conditional expectation with respect to $\mathcal{F}_T^{X^{(1)}}$ and using independence we 
deduce that $\fc_n^{(1)}$ converges in $L^2_{X^{(1)}}$ to some random variable $\fc^{(1)} \in H_{X^{(1)}}$, 
and using similar analysis for $X^{(2)}$ we note that each element $\fc \in H_X$ can be represented as 
$\fc = \fc^{(1)} +\fc^{(2)}$ for some $\fc^{(1)} \in H_{X^{(1)}}$ and $\fc^{(2)} \in H_{X^{(2)}}$. 
As an immediate consequence of the independence of $X^{(1)}$ and $X^{(2)}$, we then observe that
$$
CM_X=\left\{ \inth_1+\inth_2;\; \inth_1\in CM_{X^{(1)}},\;\inth_2\in CM_{X^{(2)}} \right\}.
$$
Now, the assertion is a direct consequence of Theorem \ref{thm:quasi_equivalent_to_check} together with the assumption that $X^{(1)}$ and $X^{(2)}$ have indefinite Wiener integrals.
\end{proof}

\section*{Acknowledgements}
Lauri Viitasaari was partially funded by Emil Aaltonen Foundation.

\appendix

\section{On Example \ref{exa:BSDE}}
\label{sec:A}
Let $X$ be a fractional Brownian motion with Hurst index $H<\frac14$. 
We claim that then $t^{2H} \in \HH_X$,
which we applied in Example \ref{exa:BSDE}. 
By Pipiras and Taqqu \cite[Theorem 4.2]{piptaq}, $t^{2H}\in \HH_X$ if and only if
\begin{equation}
\label{eq:appendix_needed}
t^{2H} = t^{\frac12 - H} \left(I_{T-}^{\frac12 - H} \cdot^{H-\frac12}g(\cdot)\right)(t) 
\end{equation}
for some $g\in L^2\left([0,T],\ud t\right)$, or equivalently,
$$
t^{3H- \frac12} = \left(I_{T-}^{\frac12 - H} \cdot^{H-\frac12}g(\cdot)\right)(t).
$$
We first compute the fractional derivative $\left(D_{T-}^{\frac12 - H}\cdot^{3H-\frac12}\right)(t) = \left(I_{T-}^{H-\frac12}\cdot^{3H-\frac12}\right)(t)$. Using relation (2.19) and equation (2.46) in Samko et al. \cite{s-k-m} 
with $a=0$, $b=T$, $\alpha = H-\frac12$, $\beta=1$ and $\gamma = 3H+\frac12$
we obtain
\begin{eqnarray}
&&\left(I_{T-}^{H-\frac12} \cdot^{3H-\frac12}\right)(t) =\left(I_{0+}^{H-\frac12} \left(T-\cdot\right)^{3H-\frac12}\right)(T-t) \nonumber\\
\label{eq:appendix_g}&&= \frac{T^{3H-\frac12}}{\Gamma\left(H+\frac12\right)}(T-t)^{H-\frac12} {}_{2}F_1\left(\frac12 - 3H, 1, H+\frac12, \frac{T-t}{T}\right),
\end{eqnarray}
where $_{2}F_1(a,b,c,z)$ denotes the Gauss hypergeometric function. Furthermore, using (2.19) in Samko et al. \cite{s-k-m} repeatedly we obtain
\begin{equation}
\label{eq:appendix_reflection}
\left(I_{T-}^{\frac12 - H}I_{T-}^{H-\frac12}\cdot^{3H-\frac12}\right)(t) = \left(I_{0+}^{\frac12-H}I_{0+}^{H-\frac12}\left(T-\cdot\right)^{3H-\frac12}\right)(T-t).
\end{equation}
We next wish to apply (2.61) in Samko et al. \cite{s-k-m} in order to conclude that
\begin{equation}
\label{eq:appendix_int_der}
\left(I_{0+}^{\frac12-H}I_{0+}^{H-\frac12}\left(T-\cdot\right)^{3H-\frac12}\right)(T-t) = t^{3H-\frac12}.
\end{equation}
To this end, note that, again by (2.46) in Samko et al. \cite{s-k-m} with $a=0$, $b=T$, $\alpha =\frac12 + H$, $\beta = 1$, and $\gamma = 3H+\frac12$, we have
$$
\left(I_{0+}^{\frac12 + H}\left(T-\cdot\right)^{3H-\frac12}\right)(t) = \frac{T^{3H-\frac12}}{\Gamma\left(\frac32 + H\right)}t^{\frac12-H} {}_{2}F_1\left(\frac12 - 3H,1,\frac32 + H,\frac{t}{T}\right).
$$
As $_{2}F_1\left(a,b,c,z\right)$ is analytic for $|z|\leq 1$ and vanishes at $z=0$, whenever $c-b-a>0$, we conclude 
that $\left(I_{0+}^{\frac12 + H}\left(T-\cdot\right)^{3H-\frac12}\right)$ is summable in the sense of Definition 2.4 in \cite{s-k-m}
and vanishes at $t=0$. Hence, (2.61) in \cite{s-k-m} is indeed applicable and implies (\ref{eq:appendix_int_der}). 
Combining \eqref{eq:appendix_g}, \eqref{eq:appendix_reflection}, and \eqref{eq:appendix_int_der} we obtain that the function $g$ given by 
$$
g(t) = t^{\frac12-H}\frac{T^{3H-\frac12}}{\Gamma\left(H+\frac12\right)}(T-t)^{H-\frac12} {}_{2}F_1\left(\frac12 - 3H, 1, H+\frac12, \frac{T-t}{T}\right)
$$
satisfies \eqref{eq:appendix_needed}. To conclude we have to prove $g \in L^2\left([0,T],\ud t\right)$. 
We use Euler's transformation formula
$
_{2}F_1(a,b,c,z) = (1-z)^{c-b-a} _{2}F_1(c-a,c-b,c,z)
$
to get
$$
g(t) = t^{3H-\frac12}\frac{T^{\frac12-H}}{\Gamma\left(H+\frac12\right)}(T-t)^{H-\frac12} {}_{2}F_1\left(4H, H-\frac{1}{2}, H+\frac12, \frac{T-t}{T}\right).
$$
Now ${}_{2}F_1\left(4H, H-\frac{1}{2}, H+\frac12, \frac{T-t}{T}\right)$ is continuous since
$$
H+\frac12 - \left(H-\frac12\right)- 4H = 1-4H >0,
$$ and hence $g \in L^2\left([0,T],\ud t\right)$.

\bibliographystyle{plain}
\bibliography{co}
\end{document}